\documentclass[a4paper]{amsart}

\usepackage[utf8]{inputenc}
\usepackage[english]{babel}

\usepackage{amsthm}
\usepackage{amsmath}
\usepackage{amssymb}
\usepackage{latexsym}
\usepackage{enumerate}
\usepackage{graphicx}% Include figure files
\usepackage{amscd}
\usepackage{epsfig} %inclosure fig EPS
\usepackage{pgf}
\usepackage{tikz-cd}
\usepackage{float}
\usepackage{dsfont}
\usepackage{times}
\usepackage{enumerate}
\usepackage{mathrsfs}
\usepackage{amsfonts}
\usepackage{amssymb}
\usepackage{hyperref}

\newcommand{\omn}{\omega^{n-1}}

\newcommand{\ov}{\overline}

\newcommand{\DG}{\sqrt{\det G(t,\ov{\theta})}}
\newcommand{\DGr}{\sqrt{\det G(r,\ov{\theta})}}

\newcommand{\norm}[1]{\left\lVert#1\right\rVert}
\newcommand{\modul}[1]{\left\lvert#1\right\rvert}
\newcommand{\ep}[1]{\left<#1\right>}

\newcommand{\MW}{M_\omega^n}

\newcommand{\brm}{{\rm B}}
\newcommand{\srm}{{\rm S}}
\newcommand{\vrm}{{\rm V}}

\newcommand{\hrm}{{\rm H}}
\newcommand{\drm}{{\rm D}}

\newcommand{\SN}{\mathbb{S}^{n-1}}

\newcommand{\R}{\mathbb{R}}
\newcommand{\N}{\mathbb{N}}
\renewcommand{\S}{\mathbb{S}}

\newcommand{\E}{\mathbb{E}}

\newcommand{\bu}{\bar{u}}

\newcommand{\CA}{\mathcal{A}}

\newcommand{\inj}{{{\rm inj}\,}}

\newtheorem*{theorem*}{Theorem}

\renewcommand{\thetheoremName}

\newtheorem{theorem}{Theorem}[section]
\newtheorem{lemma}[theorem]{Lemma}
\newtheorem{proposition}[theorem]{Proposition}
\newtheorem{corollary}[theorem]{Corollary}

\theoremstyle{definition}
\newtheorem{definition}[theorem]{Definition}

\newtheorem{examples}[theorem]{Examples}

\newtheorem{remark}[theorem]{Remark}

\newtheorem{theoremA}{Theorem}

\numberwithin{equation}{section}

\newcommand{\Vol}{\operatorname{Vol}}
\newcommand{\tr}{\operatorname{tr}}

\newcommand{\Div}{\operatorname{div}}

\newcommand{\erre}{\mathbb{R}}

\newcommand{\pe}{\mathbb{P}}

\newcommand{\Hn}{\mathbb{H}^n}

\newcommand{\ene}{\mathbb{N}}

\title[First Dirichlet eigenvalue and exit time moment spectra comparisons]{First Dirichlet eigenvalue and exit time moment spectra comparisons}

\author[V. Palmer]{Vicente Palmer*}
\address{Departament de Matem\`{a}tiques- INIT,
	Universitat Jaume I, Castellon, Spain.}
\email{palmer@uji.es}
\author[E. Sarri\'on-Pedralva]{Erik Sarri\'on-Pedralva**}
\address{Departament de Matem\`{a}tiques,
	Universitat Jaume I, Castellon, Spain.}
\email{eriksarrionpredalva@gmail.com}

\thanks{* Work partially supported by the Research Program of University Jaume I Project UJI-B2018-35,  DGI -MINECO grant (FEDER) MTM2017-84851-C2-2-P, and DGI -MINECO grant (FEDER) PID2020-115930GA-I00.}
\thanks{**Work partially supported by the Research Program of University Jaume I Project UJI-B2018-35,  DGI -MINECO grant (FEDER) MTM2017-84851-C2-2-P, and DGI -MINECO grant (FEDER) PID2020-115930GA-I00, and predoctoral fellowship GVA-ESF ACIF/2019/096.}


\begin{document}
\maketitle
\begin{abstract}
	We prove explicit upper and lower bounds for the Poisson hierarchy, the averaged $L^1$-moment spectra $\{\dfrac{\mathcal{A}_k\left(B_R^M\right)}{\text{vol}\left(S_R^M\right)}\}_{k=1}^\infty$, and the torsional rigidity $\mathcal{A}_1(B^M_R)$ of a geodesic ball $B^M_R$ in a Riemannian manifold $M^n$ which satisfies that the mean curvatures of the geodesic spheres $S^M_r$ included in it, (up to the boundary $S^M_R$), are controlled by the radial mean curvature of the geodesic spheres $S^\omega_r(o_\omega)$ with same radius centered at the center $o_\omega$ of a rotationally symmetric model space $M^n_\omega$. As a consecuence, we prove a first Dirichlet eigenvalue $\lambda_1(B^M_R)$ comparison theorem and show that equality with the bound $\lambda_1(B^\omega_R(o_\omega))$, (where $B^\omega_r(o_\omega)$ is the geodesic $r$-ball in $M^n_\omega$), characterizes the $L^1$-moment spectrum $\{\mathcal{A}_k(B^M_R)\}_{k=1}^\infty$ as the sequence $\{\mathcal{A}_k(B^\omega_R)\}_{k=1}^\infty$ and vice-versa.
\end{abstract}

%%%%%%%%%%%%%%%%%%%%%%%%%%%%%%%%%%%%%%%%
%%%%%%%%%%%%%%%%%%%%%%%%%%%%%%%%%%%
\section{Introduction}\label{sec:intro}\
%%%%%%%%%%%%%%%%%%%%%%%%%%%%%%%%%
%%%%%%%%%%%%%%%%%%%%%%%%%%%%%%%

Let $(M^n,g)$ be a complete Riemannian manifold. We shall consider the Brownian motion $X_t$ in $M$ and, given $x \in M$,  its associated family of probability measures $\pe^x$ in the space of Brownian paths emanating from a point $x \in M$.
Given a smoothly bounded precompact domain $D \subseteq M$, the first exit time from $D$ is given by the quantity
$$\tau_D:= \inf \{t \geq 0 : X_t \notin D\}\, .$$

Given $x \in D$, the function $E_D: D \rightarrow \erre$ that assigns to the point $x$  the expectation of the value of the first exit time $\tau_D$ with respect $\pe^x$ is the {\em mean exit time function from x}, $E_D(x)$. We have the following characterization, (see \cite{Dy}),  of this function as a solution of the second order PDE, with Dirichlet boundary data:

\begin{equation}\label{eq:moments1}
	\begin{split}
		\Delta^M E_D+1 &=0,\,\text{ in }D,\\
		\left.E_D\right|_{\partial D} & =0,
	\end{split}
\end{equation}

\noindent where $\Delta^M$ denotes the Laplace-Betrami operator on $(M^n,g)$.

The mean exit time function is the first in a sequence of functions $\{E=u_{1,D}\;, u_{2,D}\;, ....\}$ defined in $D \subseteq M$ inductively as follows

\begin{equation}\label{poisson}
		\begin{split}
			\Delta^Mu_{1,D}+1 & =0,\,\text{ on }D,\\
			u_{1,D}\lvert_{_{\partial D}} & = 0,
		\end{split}
	\end{equation}
	
	\noindent and, for $k\geq 2$,
	
	\begin{equation}\label{poissonk}
		\begin{split}
			\Delta^Mu_{k,D}+ku_{k-1,D} & =  0,\,\, \text{on }\,\,D,\\
			u_{k,D}\lvert_{_{\partial D}} & = 0.
		\end{split}
	\end{equation}

\noindent This sequence is the so-called, (see \cite{DLD}), {\em Poisson hierarchy for} $D$.

The Poisson hierarchy of the domain $D$ determines the $L^p$-moment spectrum of $D$, which can be defined as the following sequence of integrals, (see e.g.  \cite{Mc} and references therein for a more detailed exposition about these concepts):
$$\CA_{p,k}(D):= \bigg(\int_D(u_{k,D}(x))^p dV\bigg)^{\frac{1}{p}}\, , \,\,k=1,2,...,\infty.$$

We are going to focus our study in the $L^1$-moment spectrum of $D$, $\{\CA_{1,k}(D)\}_{k=1}^\infty$ which we denote as $\{\CA_{k}(D)\}_{k=1}^\infty$ and, in particular, in its first value, $\CA_1(D)$, called the \emph{torsional rigidity of $D$} which is defined as the integral

\begin{equation}\label{eq:Apk}
	\CA_1(D)=\int_D E_D(x)\,d\sigma,
\end{equation}

\noindent where $E_D$ is the smooth solution of the  Dirichlet-Poisson equation (\ref{eq:moments1}).

The name \lq\lq torsional rigidity" comes from the fact that, when $D \subseteq \erre^2$ is a plane domain, the quantity 
$\CA_1(D)$ represents the torque required when twisting an elastic beam of uniform cross section $D$, (see \cite{PS}). A natural question consists in to opimize this quantity among all the domains having the same given area/volume in a fixed space or under some other geometrical setting. This problem is known as a {\em Saint-Venant type problem}.

The study of this variational problem in the general context of Riemannian manifolds  involves the stablishment of bounds on the torsional rigidity of a given domain $D \subseteq M$, together the determination of the domains and the spaces that shelter them where these bounds are attained, in an analogous way that the study of the Rayleigh conjecture for the fundamental tone, and the techniques involved in this analysis encompasses the use of the notion of  Schwarz symmetrization as well as the isoperimetric inequalities satisfied by the domains in question.

From the intrinsic point of view, the stablishment of bounds for the $L^p$-moment spectrum and the study of the relationship between the torsional rigidity (and more generally, the $L^1$-moment spectrum), of a domain in a Riemannian manifold $D \subseteq M$ and its Dirichlet spectrum has been explored along the last years in a number of papers, (see, among others,  \cite{BBC}, \cite{BG},\cite{BGJ},  \cite{BFNT}, \cite{DLD}, \cite{ KD}, \cite{KDM}, \cite{HuMP1}, \cite{HuMP2}, \cite{HuMP3}, \cite{Mc},\cite{Mc2},  \cite{McMe}, \cite{MP4} and the references therein). Related with this issue and  in the line of the classical Kac's question, we have the isospectrality problem, namely, to see to what extent the $L^1$- moment spectrum of a domain determines it up to isometry, (see \cite{CD} and \cite{CKD}). 

From the viewpoint of submanifold theory, we can find in the papers \cite{MP4}, \cite{HuMP1}, and  \cite{HuMP2} upper and lower bounds for the $L^1$-moment spectrum of extrinsic balls $B^M_R\cap \Sigma$, (let us denote as $B^M_R$ the geodesic $R$-ball in the manifold $M$), in submanifolds $\Sigma^m \subseteq M^n$ with controlled mean curvature $H_\Sigma$ immersed in ambient Riemannian manifolds $(M,g)$ with radial sectional curvatures $K_{(M, g)}(\frac{\partial}{\partial r}, \, )$ bounded from above or from below.

These bounds were given, on the basis of previously stablished isoperimetric inequalities, by the corresponding values for the torsional rigidity of the Schwartz symmetrization of the geodesic balls in rotationally symmetric spaces with a pole which are warped products of the form $M^n_w=[0, \infty)\times_{w} \erre^{+}$,  and we call the {\em model spaces}. As we shall see in subsection 2.2,  the model spaces $M^n_w$ are rotationally symmetric generalizations of the real space forms with constant sectional curvature $b \in \erre$, denoted as $M^n_{w_b}=\erre^n, \S^n(b),\, \text{or}\, \Hn(b)$, with 
\begin{equation*}
		\omega_b(r)=\begin{cases}
			\dfrac{1}{\sqrt{b}}\sin\left(\sqrt{b}\,r\right),& \quad\text{si }\,b>0\\
			r, & \quad\text{si }\,b=0\\
			\dfrac{1}{\sqrt{-b}}\sinh\left(\sqrt{-b}\,r\right), & \quad\text{si }\,b<0 
		\end{cases}
	\end{equation*}
	
	\noindent We shall denote as $B^\omega_r(o_\omega)$ and as $S^\omega_r(o_\omega)$ the geodesic $r$-ball centered at $o_\omega$ and the geodesic $r$-sphere, respectively, in $M^n_\omega$.

Moreover, in \cite{MP4}, \cite{HuMP1} and \cite{HuMP2}, the geometry of situations where the equality with the bounds was attained was characterized. On the other hand, in these papers it were given too {\em intrinsic} upper and lower bounds for the torsional rigidity of geodesic balls $B^M_R$ in the ambient manifold when it was assumed that $\Sigma=M$, so the extrinsic distance became the intrinsic distance and are only assumed bounds on the radial sectional curvatures of the ambient manifold $M$.

To summarize the intrinsic results obtained in \cite{MP4}, \cite{HuMP1} and \cite{HuMP2} in a couple of statements, we need the following context and notation: let us consider a complete Riemannian manifold $(M,g)$, and a geodesic $R$-ball $B^M_R(o)$ centered at $o \in M$. Let us denote as $ K_{(M, g)}(\frac{\partial}{\partial r}, \, )$ its radial, (from the center $o$) sectional curvatures, (namely, the sectional curvatures of the planes containing the radial vector field $\frac{\partial}{\partial r}$, where $r$ denotes the distance function from the point $o$). With all these notions in hand, we have the following two results. The first concerns the so-called, (see \cite{HuMP1}), {\em averaged $L^1$-moment spectrum} of a geodesic ball:

\begin{theoremA}\label{teoSec}[see \cite{MP4} and \cite{HuMP1}]
Let $(M,g)$ be a complete Riemannian manifold. Let us consider $M^n_w$  a rotationally symmetric model space and let us suppose that 
$$ K_{(M_w, g_{w})}(\frac{\partial}{\partial r}, \, )\, \geq\, (\leq)   \,K_{(M, g)}(\frac{\partial}{\partial r}, \, )\,, $$
\noindent where $K_{(M_w, g_{w})}(\frac{\partial}{\partial r}, \, )$ denotes the radial sectional curvatures of $M^n_w$ from its center point $o_\omega \in M^N_\omega$. 

Then the {\em averaged} $L^1$-moments, $\{\dfrac{\CA_k\left(\brm_R^M(o)\right)}{\Vol\left(\srm_R^M(o)\right)}\}_{k=1}^\infty$ are bounded as follows

 \begin{equation}\label{momentscomp1}
		\dfrac{\CA_k\left(\brm_R^\omega(o_\omega)\right)}{\Vol\left(\srm_R^\omega(o_\omega)\right)}\geq\,(\leq)\,\dfrac{\CA_k\left(\brm_R^M(o)\right)}{\Vol\left(\srm_R^M(o)\right)}\, .
	\end{equation}

\noindent  Equality in inequality  (\ref{momentscomp1}) for some $k_0 \geq 1$ implies that $\brm^M_R(o)$ and $\brm^\omega_R(o_\omega)$ are isometric.
\end{theoremA}

Concerning now the Torsional Rigidity, we need to assume, in addition, that the model space  $M^n_w$ is {\em balanced from above}, namely, that the isoperimetric quotient given by $$q_\omega(r)=\dfrac{\Vol\left(\brm_r^\omega(o_\omega)\right)}{\Vol\left(\srm_r^\omega(o_\omega)\right)}$$
\noindent is a non-decreasing function of $r$. This condition is satisfied by a wide range of spaces, in particular, for all real space forms of constant sectional curvature. We then obtain the following

\begin{theoremA}\label{teoSec2}[see \cite{MP4} and \cite{HuMP1}]
Let $(M,g)$ be a complete Riemannian manifold. Let us consider $M^n_w$  a rotationally symmetric model space, balanced from above, and let us suppose that 
$$ K_{(M_w, g_{w})}(\frac{\partial}{\partial r}, \, )\, \geq\, (\leq)   \,K_{(M, g)}(\frac{\partial}{\partial r}, \, ) \, ,$$
\noindent where $K_{(M_w, g_{w})}(\frac{\partial}{\partial r}, \, )$ denotes the radial sectional curvatures of $M^n_w$ from its center point. 

Then the torsional rigidity $\CA_1\left(\brm_R^M(o)\right)$ is bounded as follows

\begin{equation}\label{torsrigcomp1}
		\CA_1\left(\brm_{s(R)}^\omega(o_\omega)\right)\geq\,(\leq)\,\CA_1\left(\brm_R^M(o)\right)\, ,
	\end{equation}
	
	 \noindent where $\brm_{s(R)}^\omega(o_\omega)$ is the Schwarz symmetrization of $\brm_R^M(o)$ in the model space $(\MW,g_\omega)$.
	
\noindent Equality in inequality (\ref{torsrigcomp1})  implies that $s(R)=R$ and that $\brm^M_R(o)$ and $\brm^\omega_R(o_\omega)$ are isometric.
	
\end{theoremA}

%On the other hand, in the paper \cite{HuMP3} it was obtained an explicit expression of the first Dirichlet eigenvalue of a geodesic ball $B^w_R$  in a rotationally symmetric model space $M^n_w$ as a limit of the sequence given by the $L^1$-moment spectrum  of this geodesic ball, $\{\CA_{k}(B^w_R)\}_{k=1}^\infty$, namely, if we denote as $\lambda_1(B^w_R)$ this first Dirichlet eigenvalue, then
%\begin{equation}\label{DirMoment}
%\begin{aligned}
%\lambda_1(B^\omega_R)=\lim_{k \to \infty}\dfrac{k\CA_{k-1}\left(\brm_R^\omega\right)}{\CA_{k}\left(\brm_R^\omega\right)}
%\end{aligned}
%\end{equation}

%In the paper \cite{BGJ}, this equality was extended to any precompact domain $\Omega \subseteq M$, namely
%\begin{equation}\label{DirMoment2}
%\begin{aligned}
%\lambda_1(\Omega)=\lim_{k \to \infty}\dfrac{k\CA_{k-1}\left(\Omega\right)}{\CA_{k}\left(\Omega\right)}
%\end{aligned}
%\end{equation}

As a consequence of the bounds for the $L^1$-moment spectrum stated in Theorem \ref{teoSec}, and the proof of Theorem 1.1 in \cite{McMe}, (where it was presented a formula for the first Dirichlet eigenvalue of a precompact domain $D$ in a Riemnnian manifold $M$, $\lambda_1(D)$, in terms of its $L^1$-moment spectrum, $\{\CA_k(D)\}_{k=1}^\infty$), it was obtained the following version of Cheng's eigenvalue comparison theorem:

\begin{theoremA}\label{ChengSec}[see \cite{HuMP3},\cite{Cheng75-1}, \cite{Cheng75-2}]

Let $(M,g)$ be a complete Riemannian manifold. Let us denote as $ K_{(M, g)}(\frac{\partial}{\partial r}, \, )$ its radial sectional curvatures and as  $ Ricc_{(M, g)}(\frac{\partial}{\partial r}, \frac{\partial}{\partial r})$ its radial Ricci curvatures at any point. 

Let us consider $M^n_w$  a rotationally symmetric model space and let us suppose that 
\begin{equation*}
\begin{aligned}
 K_{(M_w, g_{w})}(\frac{\partial}{\partial r}, \, )\, &\geq \, K_{(M, g)}(\frac{\partial}{\partial r}, \, )\, ,\\
\bigg( \text{or that} \, \,\,Ricc_{(M_w, g_{w})}(\frac{\partial}{\partial r}, \frac{\partial}{\partial r})\,&\leq \, Ricc_{(M, g)}(\frac{\partial}{\partial r}, \frac{\partial}{\partial r})\bigg)\, ,
\end{aligned}
\end{equation*}

\noindent where $K_{(M_w, g_{w})}(\frac{\partial}{\partial r}, \, )$ and $Ricc_{(M_w, g_{w})}(\frac{\partial}{\partial r}, \frac{\partial}{\partial r})$ denotes the radial sectional and Ricci curvatures of $M^n_w$ at its center point. 

Then
$$\lambda_1(B^{\omega}_R(o_\omega)) \,\leq\,(\geq)  \, \lambda_1(B^M_R(o)).$$

\noindent for all $R < inj(o) \leq inj(o_w)$.

Equality in any of these inequalities holds if and only if the geodesic balls $\brm^M_R(o)$ and $\brm^\omega_R(o_\omega)$ are isometric.
\end{theoremA}

On the other hand,  in the paper \cite{Mc2}, P. McDonald showed that, given a precompact domain $D \subseteq M$ in a complete Riemannian manifold $M$ such that satisfies the inequalities $\CA_k(D) \leq \CA_k(D^*)$ where $D^*$ is the Schwarz symmetrization of $D$ in a constant curvature space form $M_{\omega_b}$, then we have the inequality $\lambda_1(D^*) \leq \lambda_1(D)$, (see Theorem 1 in \cite{Mc2}). 

%In this paper \cite{Mc2} it is presented the following notion of {\em determination} of a Riemannian invariant  $I(D)$ defined on the precompact domain $D\subseteq M$ by the $L^1$-moment spectrum of $D$: we say that $\{\CA_k(D)\}_{k=1}^\infty$ {\em determines} the invariant $I(D)$ if and only if when $\{\CA_k(D)\}_{k=1}^\infty=\{\CA_k(D')\}_{k=1}^\infty$, then $I(D)=I(D')$. With this definition, in \cite{Mc2} is proved that the $L^1$-moment spectrum of a precompact domain $D$  determines its heat content. We shall return on this notion in relation with our results.

Following with versions of Cheng's result, in the paper \cite{BM}, the authors proved that Cheng's eigenvalue comparison is still valid assuming bounds on the mean curvature of (intrinsic) distance spheres, a weaker hypothesis (as we shall see below), than the bounds on the sectional curvatures of the manifold:  

\begin{theoremA}[see \cite{BM}]\label{teoPacelli}
Let $B^M_R \subseteq M^n$ and $B^{w_{b}}_R$ be geodesic $R$-balls in a Riemannian manifold $(M,g)$ and in the real space form with constant sectional curvatures $b\in \erre$, $M^n_{w_b}$, respectively, both within the cut locus of their centers and let $(t, \ov{\theta}) \in (0,R]\times\mathbb{S}^{n-1}_1$ be the polar coordinates for $B^M_R$ and $B^{w_b}_R$. 

Then, if $H_{S^M_t}(t,\ov{\theta})$ and $H_{S^{w_b}_t}(t)$ are the, (pointed inward), mean curvatures of the distance spheres $S^M_t$ in $M$ and $S^{w_b}_t$ in the real space form of constant curvature $M^{w_b}$ respectively and we assume that
$$H_{S^{w_b}_t}(t) \, \leq\, (\geq)    \, H_{S^M_t}(t,\ov{\theta})\,\,\forall t \leq R\,\,\forall \ov{\theta} \in \SN_1$$
\noindent we have that
$$\lambda_1(B^{\omega_b}_R(o_\omega)) \, \leq\, (\geq)  \, \lambda_1(B^M_R(o)).$$

Equality in any of these inequalities holds if and only if $H_{S^M_t}(t,\ov{\theta}) = H_{S^{w_b}_t}(t)\,\,\forall t \leq R\,\,\forall \ov{\theta} \in \SN_1.$
\end{theoremA}

Its proof relies in Barta's Lemma and the expresion of the Laplacian of the first Dirichlet eigenfunction  in polar coordinates.  It is precisely from this intrinsic expression that the use, as hypotheses, of bounds on the mean curvature of distance spheres comes from.

Therefore, it can be said that the results we are going to present in the following paper are inspired, by one hand,  by the intrinsic bounds for the torsional rigidity and the $L^1$-moment spectrum of the geodesic balls and by the estimation of $\lambda_1(B^M_R)$ obtained in the  papers \cite{MP4}, \cite{HuMP1}, \cite{HuMP2} and \cite{HuMP3}, and by the other hand, by the weaker restrictions on the mean curvatures of geodesic spheres assumed in \cite{BM} as well as the comparisons for the $L^1$-moment spectrum and the first Dirichlet eigenvalue given  in \cite{Mc} and \cite{Mc2}.

%%%%%%%%%%%%%%%%%%%%%%%%%%%%%%%%
\subsection{A glimpse at our results}\
%%%%%%%%%%%%%%%%%%%%%%%%%%%%%%

We consider along this paper, a complete Riemannian manifold $(M^n,g)$ and a rotationally symmetric model space $(\MW, g_w)$, with center $o_w$, and we shall assume that  given $o \in M$ a point in $M$, the injectivity radius  of $o\in M$ satisfies $inj(o) \leq \inj(o_w)$. Let us fix $R < inj(o) \leq inj(o_w)$ assuming that the pointed inward mean curvatures of metric $r$-spheres satisfies
	
	\begin{equation*}
	\begin{aligned}
		\hrm_{\srm_R^\omega(o_\omega)}\,&\leq\, \hrm_{\srm_R^M(o)}\quad\text{for all}\quad 0 < r \leq R\\
		\bigg( \text{or that }\,\,\,\,\,\hrm_{\srm_R^\omega(o_\omega)}\,&\geq\, \hrm_{\srm_R^M(o)}\quad\text{for all}\quad 0 < r \leq R\bigg)
		\end{aligned}
	\end{equation*}

These hypotheses are the same than the conditions assumed in \cite{BM}, and constitutes a more general assumption than the bounds for the sectional and the Ricci curvatures in Theorems \ref{teoSec} and \ref{ChengSec}, as we shall see in next Subsection \ref{ex}. On the other hand, they implies, in its turn,  the following isoperimetric conditions satisfied by the geodesic $r$- balls with $r \leq R$ in the complete Riemannian manifold $M$, 
\begin{equation} \label{hypisop}
	\dfrac{\Vol\left(\brm_r^\omega(o_\omega)\right)}{\Vol\left(\srm_r^\omega(o_\omega)\right)}\geq\,(\leq)\,\dfrac{\Vol\left(\brm_r^M(o)\right)}{\Vol\left(\srm_r^M(o)\right)}\quad\text{for all}\quad 0<r \leq R.
\end{equation}
		
		Concerning the use of isoperimetric inequalities, (no exactly the given in (\ref{hypisop})), in the study of the relation between the moments spectrum and the Dirichlet spectrum, we refer to the paper \cite{Mc}.
		
		Under these restrictions on the mean curvatures of geodesic spheres we have obtained all the results in this paper, the most important of which are Proposition \ref{prop3.2}, Theorem \ref{teo:MeanComp} and Corollary \ref{cor:MeanComp} in Section \ref{sec:MeanComp}, Theorem \ref{teo:TowMomComp}, Corollary \ref{isoptors} and Theorem \ref{teo:2.1} in Section \ref{sec:TorRidCom1}, and Theorem \ref{th_const_below2} and Corollary \ref{th_const_below3} in Section \ref{sec:TorRidCom2}. 
		
		%We must remark, however, that Proposition \ref{prop:ineqpsiE} and Theorem \ref{teo:2.1} in Subsection \ref{subsec:TRcomp} holds under the weaker assumption that the geodesic $r$- balls, (and its boundaries), with $r \leq R$, $S_r^M(o)\subseteq B_r^M(o)  \subseteq B^M_R(o)$, satisfies the isoperimetric condition (\ref{hypisop}). This isoperimetric condition  gives another geometrical flavour to the inequality between the mean curvatures of geodesic spheres and the corresponding model space. It is worthy to say that a similar (although non equivalent) isoperimetric condition, was used by P. MacDonald in an early paper, (see \cite{Mc}), where it was also studied the relation between the moments spectrum and the Dirichlet spectrum.

	We are going to present in the following statements of Theorem \ref{newresult1}, Theorem \ref{newresult2} and Theorem\ref{ChengMean}  summaryzed versions of some of our results concerning bounds on the Poisson hierarchy, the $L^1$-moments spectrum and the first Dirichlet eigenvalue of the geodesic balls $B^M_R$ in Sections \ref{sec:TorRidCom1} and \ref{sec:TorRidCom2}, in order to see that they are a generalization of those presented in Theorem \ref{teoSec} and in Theorem \ref{ChengSec}. 
	
	The techniques used in the proof of these results are basically the same than in the cited papers \cite{MP4}, \cite{HuMP1}, and \cite{HuMP2}, but now with the  intrinsic point of view as the main perspective. These techniques encompasses the use of the formula of the Laplacian of the mean exit time function in polar coordinates, the application of the Maximum principle, the properties of the Schwartz symmetrization of the ball $B^M_R$ and the explicit expression of the first Dirichlet eigenvalue of a geodesic ball $B^w_R$  in a rotationally symmetric model space $M^n_w$ as a limit of the sequence given by the $L^1$-moment spectrum  of this geodesic ball, $\{\CA_{k}(B^w_R)\}_{k=1}^\infty$,  obtained in the paper \cite{HuMP3}. This formula for $\lambda_1(B^\omega_R(o_\omega))$ was subsequently extended in the paper \cite{BGJ} to any precompact domain $\Omega \subseteq M$, namely
\begin{equation}\label{DirMoment2}
\begin{aligned}
\lambda_1(\Omega)=\lim_{k \to \infty}\dfrac{k\CA_{k-1}\left(\Omega\right)}{\CA_{k}\left(\Omega\right)}.
\end{aligned}
\end{equation}

In fact, the presence of the mean curvature of geodesic spheres $H_{S^M_{t}} $ in the expresion of the Laplacian operator in polar coordinates has played a key role in the stablishment of our hypotheses, (as it is obvious), and also in the analysis of the equality with the bounds in all of our comparisons.
	
	Concerning this analysis of the equality case, an important notion which appears in next Theorems \ref{newresult1}, \ref{newresult2} and \ref{ChengMean}, (in fact, along all the equality discussions in the paper), is the concept of {\em determination} of a Riemannian invariant defined on the geodesic balls by its $L^1$-moment spectrum, its  $L^1$-averaged moment spectrum or its Torsional Rigidity, in a way which, altthough it is not exactly the same, it has been directly inspired by P. McDonald in \cite{Mc2}. 
	
	In the paper \cite{Mc2} it is presented the notion of {\em determination} of a Riemannian invariant  $I(D)$ defined on the precompact domain $D\subseteq M$ by the $L^1$-moment spectrum of $D$: we say that $\{\CA_k(D)\}_{k=1}^\infty$ {\em determines} the invariant $I(D)$ if and only if when $\{\CA_k(D)\}_{k=1}^\infty=\{\CA_k(D')\}_{k=1}^\infty$, then $I(D)=I(D')$. With this definition, in \cite{Mc2} is proved that the $L^1$-moment spectrum of a precompact domain $D$  determines its heat content. 
	
	We shall see in the following Theorem \ref{newresult1} and Theorem \ref{newresult2} that, under our hypotheses, the Torsional Rigidity $\CA_1(B^M_R)$ and any individual averaged moment   $\dfrac{\CA_{k_0}\left(\brm_R^M(o)\right)}{\Vol\left(\srm_R^M(o)\right)}$ determines the Poisson hierarchy, the volume, the $L^1$-moment spectrum and the first Dirichlet eigenvalue of the ball $B^M_R$, in the following sense:
	
	 When $\CA_1(B^M_R(o))=\CA_1(B^\omega_{s(R)}(o_\omega))$, or there exists $k_0 \geq 1$ such that 
	$$\dfrac{\CA_{k_0}\left(\brm_R^M(o)\right)}{\Vol\left(\srm_R^M(o)\right)}=\dfrac{\CA_{k_0}\left(\brm_R^\omega(o_\omega)\right)}{\Vol\left(\srm_R^\omega(o_\omega)\right)},$$
	\noindent then $s(R)=R$ and the Poisson hierarchy, the volume, the $L^1$-moment spectrum and the first Dirichlet eigenvalue of the ball $B^M_R$ is the same than the corresponding values for the geodesic ball $B^\omega_R(o_\omega)$ in the model space $M^n_\omega$.
	
	With all these previous considerations, we present  the following:
	
\begin{theorem}\label{newresult1}[see Corollary \ref{isoptors}]

Let us consider a complete Riemannian manifold $(M^n,g)$ and a rotationally symmetric model space $(\MW, g_w)$, with center $o_w$, and we shall assume that  given $o \in M$ a point in $M$, the injectivity radius  of $o\in M$ satisfies $inj(o) \leq \inj(o_w)$. Let us fix $R < inj(o) \leq inj(o_w)$ assuming that the pointed inward mean curvatures of metric $r$-spheres satisfies
	
	\begin{equation}\label{eq:meancurvatureconditions}
		\hrm_{\srm_R^\omega(o_\omega)}\leq\,(\geq)\, \hrm_{\srm_R^M(o)}\quad\text{for all}\quad 0 < r \leq R.
	\end{equation}
Then, for all $k\geq 1$,
	
	\begin{equation}\label{momentscomp2}
		\dfrac{\CA_k\left(\brm_R^\omega(o_\omega)\right)}{\Vol\left(\srm_R^\omega(o_\omega)\right)}\geq\,(\leq)\,\dfrac{\CA_k\left(\brm_R^M(o)\right)}{\Vol\left(\srm_R^M(o)\right)}.
	\end{equation}
	
	 Equality in any of  inequalities \ref{momentscomp2} for some $k_0 \geq 1$ implies  that 
	 $$\hrm_{\srm_R^\omega(o_\omega)}=\,\hrm_{\srm_R^M(o)}\,\,\,\text{for all }0<r\leq R$$
	 \noindent and hence, we have the equalities
\begin{enumerate}
	\item Equality $\bu_{k,R}^{\omega}=u_{k,R}$ on $B^M_R(o)$ for all $k\geq 1$, and hence, equality $\bu_{k,r}^{\omega}=u_{k,r}$ on $B^M_r(o)$ for all $k\geq 1$ and for all $0 < r \leq R$.
	\item Equalities $\Vol\left(\brm_r^\omega(o_\omega)\right)=\Vol\left(\brm_r^M(o)\right)$ and $\Vol\left(\srm_r^\omega(o_\omega)\right)=\Vol\left(\srm_r^M(o)\right)\,\,\,\text{for all }\, 0<r \leq R$.
	\item Equalities $\CA_k\left(\brm_r^\omega(o_\omega)\right)=\CA_k\left(\brm_r^M(o)\right)$,  for all $k\geq 1$ and for all $0<r \leq R$.
	\item Equality $\lambda_1(B^M_r(o))=\lambda_1(B^w_r(o_\omega))$ for all $0<r \leq R$.

\noindent Namely, {\em one} value of $\dfrac{\CA_k\left(\brm_R^M(o)\right)}{\Vol\left(\srm_R^M(o)\right)}$ for some $k \geq 1$ determines the Poisson hierarchy, the volume, the $L^1$-moment spectrum and the first Dirichlet eigenvalue of the ball $B^M_r(o)$ for all $0<r \leq R$.
\end{enumerate}

\end{theorem}

Our second result is a comparison for the Torsional Rigidity of the ball $B^M_R$, and, as in Theorem \ref{teoSec2}, we need that the model space $M^n_\omega$ used in the comparison be {\em balanced from above}.

\begin{theorem}\label{newresult2}[see Theorem \ref{teo:2.1} ]

Let us consider a complete Riemannian manifold $(M^n,g)$ and a balance from above rotationally symmetric model space $(\MW, g_w)$, with center $o_w$, and we shall assume that  given $o \in M$ a point in $M$, the injectivity radius  of $o\in M$ satisfies $inj(o) \leq \inj(o_w)$. Let us fix $R < inj(o) \leq inj(o_w)$ assuming that the pointed inward mean curvatures of metric $r$-spheres satisfies
	
	\begin{equation}\label{eq:meancurvatureconditions}
		\hrm_{\srm_R^\omega(o_\omega)}\leq\,(\geq)\, \hrm_{\srm_R^M(o)}\quad\text{for all}\quad 0 < r \leq R.
	\end{equation}
Then

 \begin{equation}\label{torsrigcomp21}
		\CA_1\left(\brm_{s(R)}^\omega(o_\omega)\right)\geq\,(\leq)\,\CA_1\left(\brm_R^M(o)\right)\, ,
	\end{equation}
	
	\noindent where $\brm_{s(R)}^\omega(o_\omega)$ is the Schwarz symmetrization of $\brm_R^M(o)$ in the model space $(\MW,g_\omega)$.
	
	 Equality in any of  inequalities (\ref{torsrigcomp21}) implies the equality among the radius $s(R)=R$ and that 
	 $$\hrm_{\srm_R^\omega(o_\omega)}=\,\hrm_{\srm_R^M(o)}\,\,\,\text{for all }0<r\leq R$$
	 \noindent and hence, we have the equalities
\begin{enumerate}
	\item Equality $\bu_{k,R}^{\omega}=u_{k,R}$ on $B^M_R(o)$ for all $k\geq 1$, and hence, equality $\bu_{k,r}^{\omega}=u_{k,r}$ on $B^M_r(o)$ for all $k\geq 1$ and for all $0 < r \leq R$.
	\item Equalities $\Vol\left(\brm_r^\omega(o_\omega)\right)=\Vol\left(\brm_r^M(o)\right)$ and $\Vol\left(\srm_r^\omega(o_\omega)\right)=\Vol\left(\srm_r^M(o)\right)\newline \text{for all } 0<r \leq R$.
	\item Equalities $\CA_k\left(\brm_r^\omega(o_\omega)\right)=\CA_k\left(\brm_r^M(o)\right)$,  for all $k\geq 1$ and for all $0<r \leq R$.
	\item Equality $\lambda_1(B^M_r(o))=\lambda_1(B^w_r(o_\omega))$ for all $0<r \leq R$.
	
	\noindent Namely, the Torsional Rigidity determines the Poisson hierarchy, the volume, the $L^1$-moment spectrum and the first Dirichlet eigenvalue of the ball $B^M_r(o)$ for all $0<r \leq R$.

\end{enumerate}

\end{theorem}

As a consequence of the proof of Theorem 1.1 in \cite{McMe}, Theorem \ref{isoptors},  and volume inequalities given in Theorem \ref{teo:MeanComp}, we have the following Cheng's Dirichlet eigenvalue comparison, following \cite{BM}, (Theorem \ref{th_const_below2} in Section 5). In this case, we have proved that the first Dirichlet eigenvalue of $B^M_R$ determines its Poisson hierarchy, its  volume and its $L^1$-moment spectrum.

\begin{theorem}\label{ChengMean}[see Theorem \ref{th_const_below2}]
Let $(M^n,g)$ be a complete Riemannian manifold and let $(\MW,g_\omega)$ be a rotationally symmetric model space with center $o_\omega \in \MW$. Let $o \in M$ be a point in $M$ and let us suppose that  $inj(o) \leq inj(o_\omega)$. Let us consider a metric ball $B^M_R(o)$, with $R < inj(o) \leq inj(o_w)$. Let us suppose moreover that the pointed inward mean curvatures of the geodesic spheres in $M$ and $M_{\omega}$ satisfies

	\begin{equation}\label{eq:meancurvatureconditions12}
		\hrm_{\srm_R^\omega(o_\omega)}\leq\,(\geq)\, \hrm_{\srm_R^M(o)}\quad\text{for all}\quad 0 < r \leq R.
	\end{equation}
	
	Then we have the inequalities
 \begin{equation}\label{ineqleq_submanifold2}
\lambda_1 (B^\omega_R(o_\omega))\, \leq\,(\geq)   \,   \lambda_1(B^M_R(o))\, .\end{equation}
%\noindent where $B^\omega_R(o_\omega)$ is the geodesic ball in $M^n_\omega$.
 
\noindent  Equality in any of these inequalities implies that 
$$\hrm_{\srm_R^\omega(o_\omega)}=\,\hrm_{\srm_R^M(o)}\,\,\,\text{for all }0<r\leq R$$
\noindent and hence, we have the equalities

\begin{enumerate}
	\item Equality $\bu_{k,R}^{\omega}=u_{k,R}$ on $B^M_R(o)$ for all $k\geq 1$, and hence, equality $\bu_{k,r}^{\omega}=u_{k,r}$ on $B^M_r(o)$ for all $k\geq 1$ and for all $0 < r \leq R$.
	\item Equalities $\Vol\left(\brm_r^\omega(o_\omega)\right)=\Vol\left(\brm_r^M(o)\right)$ and $\Vol\left(\srm_r^\omega(o_\omega)\right)=\Vol\left(\srm_r^M(o)\right) \newline \text{for all }\, 0<r \leq R$.
	\item Equalities $\CA_k\left(\brm_r^\omega(o_\omega)\right)=\CA_k\left(\brm_r^M(o)\right)$,  for all $k\geq 1$ and for all $0<r \leq R$.

	\end{enumerate}
\noindent Namely, the first Dirichlet eigenvalue determines  the Poisson hierarchy, the volume, and  the $L^1$-moment spectrum of the ball $B^M_r(o)$ for all $0<r \leq R$.

  \end{theorem}

%The techniques used in the proof of these results are basically the same than in the cited papers \cite{MP4}, \cite{HuMP1}, and \cite{HuMP2}, but now with the  intrinsic point of view as the main perspective. We use the expresion of the Laplacian of the mean exit time function in polar coordinates, the Maximum principle and the Schwartz symmetrization. 

%All this machinery have allowed us to prove, in addition to the bounds stated in assertions (1) and (2) in Theorem \ref{newresult1}, the fact that the torsional rigidity  and the averaged moment spectrum  of geodesic balls determines, (in the sense explicited in  \cite{Mc2}), its first Dirichlet eigenvalue. Namely, this determination  must be understood such as it has been described in assertion (3) in  Theorem \ref{newresult1}:   equality in inequalities (\ref{torsrigcomp2}) or in inequalities (\ref{momentscomp2}) for some fixed $k _0\geq 1$ implies that the Poisson hierarchy of the ball $B^M_R$ is equal to the Poisson hierarchy of the $\omega$-model ball $B^\omega_R$, that  the $L^1$-moment spectrum of the ball $B^M_R$ equals to the $L^1$-moment spectrum of the ball $B^\omega_R$ and finally, that the first Dirichlet eigenvalue of the ball $B^M_R$ is the first Dirichlet eigenvalue of the ball $B^\omega_R$.

%Moreover,  in Theorem \ref{ChengMean} we show that the reciprocal statement is also true, namely, that the first Dirichlet eigenvalue of geodesic balls determines its moment spectrum and its Poisson hierarchy.

The characterizations of equalities in both theorems are ultimately based on a rigidity property satisfied by the Poisson hierarchy for $B^M_R$, $\{u_{k,R}\}_{k=1}^\infty$ in a Riemannian manifold $M$ under the hypotheses depicted above. This rigidity property can be summarized  by saying that the value at {\em one point} $ p \in B^M_R$ of one of the functions $u_{k,R}$ of the Poisson hierarchy determines it {\em entirely} on the geodesic ball $B^M_R$, ( see Proposition \ref{prop3.2} and  assertions (3) and (4) in Theorem \ref{teo:TowMomComp} in Section \ref{sec:TorRidCom1}).

% the following results:
%\begin{itemize}
%\item Bounds for the mean exit time function and the isoperimetric inequalities stated in Theorem \ref{teo:MeanComp}  and Corollaries \ref{cor:MeanComp} and \ref{cor:MeanComp2} in Section 3.
%\item  Bounds for the Poisson hierarchy and the averaged $L^1$-moment spectrum of the geodesic balls, stated in Theorems \ref{teo:TowMomComp}, and Corollary \ref{isoptors} in Subsection 4.1. As a consequence of these results, together equality (\ref{DirMoment}) and its extension to precompact domains in the manifold, it has been showed in Theorem \ref{lastTheorem2} that any of the averaged moments of the averaged $L^1$-moment spectrum of an geodesic ball determines its first Dirichlet eigenvalue.
%\item Bounds for the torsional rigidity of the geodesic balls, stated in Theorem \ref{teo:2.1} in Subsection 4.2. As a consequence of this comparison, it has been proved in Theorem \ref{lastTheorem1}  that the torsional rigidity of geodesic balls determines its first Dirichlet eigenvalue.
%\item The Cheng's first Dirichlet eigenvalue comparison alluded above, (Theorem  \ref{th_const_below2} in Section 5).
%\end{itemize}

%%%%%%%%%%%%%%%%%%%%%%%%%%%%%%%%
\subsection{Example}\label{ex}\
%%%%%%%%%%%%%%%%%%%%%%%
	
	We remark that, under the bounds on the sectional curvatures of the manifold as hypothesis and if we have the equality with the corresponding bound in the model space of any of our invariants defined on the geodesic ball $B^M_R$, (namely, the Poisson hierarchy, the averaged $L^1$-moment spectrum, or the torsional rigidity), then $B^M_R$ is isometric to the geodesic balls in the model space, $B^w_R \subseteq M^n_w$. However, the equality of the mean curvature of distance spheres in the Riemannian manifold $M$, with its radial bound given by the mean curvature of distance spheres in the model space $M_\omega$ does not imply the isometry  among the geodesic balls, as in the previous case.
	
	This observation is coherent with the fact that bounds on the sectional curvatures of the manifold implies bounds for the mean curvature of its geodesic spheres, namely, if (M,g) is a Riemannian manifold with radial sectional curvatures 
	$$K_{sec,g}(\frac{\partial}{\partial r}, \, ) \leq \, (\geq)\,K_{sec,g_{w}}(\frac{\partial}{\partial r}, \, )=-\frac{w''(r)}{w(r)}$$
	\noindent  then we have that 
	$$H_{S^M_{r}}  \geq\,(\leq)\, H_{S^w_{r}}=\frac{w'(r)}{w(r)}\, .$$

	These implications follows from the observation that the mean curvature of geodesics spheres is the Laplacian of the distance from its center in the manifold, (see Proposition \ref{lapmean}), together the Hessian comparison analysis of the distance function as it can be found in \cite{GreW}, \cite{MM} or \cite{Pa3}.
	
	However, in the paper \cite{BM}, the authors exhibit in Example 3.1 in Section 3, smooth complete and rotationally symmetric metrics $g$ on $\erre^n$ with radial sectional curvatures bounded from below, $K_{sec, g}(\frac{\partial}{\partial r},\,) \geq b$ outside a compact set and such that the distance spheres $S^{(\erre^{n}, g)}_t$ have mean curvature $H_{S^{(\erre^{n}, g)}_{t}} \geq H_{S^{w_{b}}_{t}}$.
	
	In the following, we are going to present a new example which shows that bounds on the mean curvature of geodesic spheres of the manifold does not imply that the sectional curvatures of the manifold are controlled.

 Let $(\mathbb{R}^2,g)$ be a Riemannian manifold such that its metric tensor expressed in polar coordinates is given by $g=dr^2+\omega^2(r,\theta)d\theta^2$, where $\omega:\mathbb{R}^2\to\mathbb{R}$ is a positive smooth function given by

\begin{equation}\label{eq:metricaexemple}
	\omega(r,\theta)=r\left(1+\dfrac{r^2}{1+r^2\cos^2\theta}\right).
\end{equation}

On the other hand, we consider as a model space the simply connected space form $(\mathbb{R}^2,g_{{\rm can}})$ of constant sectional curvature $b=0$.

We are going to see that the mean curvatures of the geodesic spheres ${\rm S}^{(\erre^{2}, g)}_t(\vec{0})$ of $(\mathbb{R}^2,g)$ centered at $\vec{0}$ with radius $t$, are bounded from below by the mean curvatures of the geodesic spheres ${\rm S}_t^{\omega_0}(\vec{0})$ of $(\mathbb{R}^2,g_{\rm can})$ centered at $\vec{0}$ with the same radius, namely, that 
$${\rm H}_{{\rm S}^{(\erre^{2}, g)}_t}\geq{\rm H}_{{\rm S}_t^{\omega_0}}$$
%\noindent  for all $t>0$ and for any point $(t,\theta)\in{\rm S}^{(\erre^{2}, g)}_t(\vec{0}) \subseteq \erre^2$.

As ${\rm H}_{{\rm S}^{(\erre^{2}, g)}_t}(t,\theta)=\dfrac{\frac{\partial\omega}{\partial t}(t,\theta)}{\omega(t,\theta)}$ and we have
\begin{equation*}
	\begin{split}
		\dfrac{\partial \omega}{\partial t}(t,\theta) & =1+\dfrac{t^2}{1+t^2\cos^2\theta}+\dfrac{2t^2}{(1+t^2\cos^2\theta)^2}		
	\end{split}
\end{equation*}
we obtain

\begin{equation*}
	%\begin{split}
		{\rm H}_{{\rm S}^{(\erre^{2}, g)}_t}(t,\theta)  =\dfrac{\frac{\partial\omega}{\partial t}(t,\theta)}{\omega(t,\theta)}=
		 \dfrac{1}{t}+\dfrac{2t}{(1+\frac{t^2}{1+t^2\cos^2\theta})(1+t^2\cos^2\theta)^2}
	%\end{split}
\end{equation*}

\noindent But  
$$\dfrac{2t}{(1+\frac{t^2}{1+t^2\cos^2\theta})(1+t^2\cos^2\theta)^2}\geq 0\,\,\text{for all}\,\,(t,\theta)\in (0,+\infty)\times[0,2\pi)$$ 

Hence we have that 
$${\rm H}_{{\rm S}^{(\erre^{2}, g)}_t}(t,\theta)\geq \dfrac{1}{t}={\rm H}_{{\rm S}_t^{\omega_b}}(t)\,\,\text{for all}\,\,(t,\theta)\in (0,+\infty)\times[0,2\pi)$$

Now, let us consider the unique $2$-plane tangent to a point $(t,\theta)\in\R^2$ generated by the coordinate vector fields $\left\{\frac{\partial}{\partial r},\frac{\partial}{\partial\theta}\right\}$. We are going to compute the sectional curvature of $(\mathbb{R}^2,g)$ at this point and we will see that  it is not bounded by the corresponding sectional curvature of $(\mathbb{R}^2,g_{\rm can})$, i.e., we will show that $K_{sec,g}(t,\theta)$ is not bounded from below by $0$.

As $K_{sec,g}(t,\theta)=-\dfrac{\frac{\partial^2\omega}{\partial t^2}(t,\theta)}{\omega(t,\theta)}$ then, it is straightforward to check that

\begin{equation*}
	%\begin{split}
		K_{sec,g}(t,\theta) =-\dfrac{\frac{\partial^2\omega}{\partial t^2}(t,\theta)}{\omega(t,\theta)}=
		\dfrac{2(t^2\cos^2\theta-3)}{(1+t^2\cos^2\theta)^2(1+t^2+t^2\cos^2\theta)}
	%\end{split}	
\end{equation*}

Thus, for $\theta=0$, we have that

\begin{equation*}
	K_{sec,g}(t,0)=\dfrac{2(t^2-3)}{(1+t^2)^2(1+2t^2)}
\end{equation*}

This shows that there are points $(t,\theta)\in\R^2$ where  the sectional curvature of $(\mathbb{R}^2,g)$ is bounded either below or above by $0$ which is the sectional curvature of $(\mathbb{R}^2,g_{{\rm can}})$.

%We end this subsection observing that, as well as the bounds for the sectional curvatures implies the bounds of the mean curvatures of geodesic spheres, it can be proved, (see Remark \cite{} in Section \ref{sec:MeanComp}), that the bounds of the mean curvatures of geodesic spheres implies that these geodesic spheres and its correponding geodesic balls satisfies an isoperimetric inequality. This isoperimetric condition, which gives a more geometrical flavour to the inequality between the mean curvatures of geodesic spheres and the corresponding model space, was used by P. MacDonald in an early paper, (see \cite{Mc}), where it was studied the relation between the moments spectrum and the Dirichlet spectrum.
%%%%%%%%%%%%%%%%%%%%%%%%%%%%%%%%,
\subsection{Outline}
%%%%%%%%%%%%%%%%%%%%%%%
After the Introduction, Section \ref{sec:preliminars} is devoted to the presentation of preliminary concepts, including the rotationally symmetric model spaces used to construct the bounds and the notion of Schwarz symmetrization based on these models. We have stated and proved, for the sake of completeness, all the properties of these symmetrizations we need in our context. Next Section \ref{sec:MeanComp} deals with the properties of the mean exit time function defined on the geodesic $R$-balls in a complete Riemannian manifold satisfying our hypotheses and its relation with its volume and the isoperimetric inequalities satisfied by these domains, (Proposition \ref{prop3.2}, Theorem \ref{teo:MeanComp}, Corollary \ref{cor:MeanComp} and Corollary \ref{cor:MeanComp2}) . In Section \ref{sec:TorRidCom1} we have stablished bounds for the Poisson hierarchy and the averaged $L^1$-moment spectrum of a geodesic $R$-ball under our restricitions (Theorem \ref{teo:TowMomComp} and  Corollary \ref{isoptors}) , and we have bounded too the Torsional Rigidity of a geodesic $R$-ball by means its Schwarz symmetrization, (Theorem  \ref{teo:2.1}). Finally, in Section \ref{sec:TorRidCom2}, we have showed a Cheng's comparison for the first Dirichlet eigenvalue of geodesic balls, (Theorem \ref{th_const_below2}), and we have stablished the relation between the first Dirichlet eigenvalue of geodesic balls,  its $L^1$-moment spectrum and its Poisson hierarchy in Corollary \ref{th_const_below3}.

%%%%%%%%%%%%%%%%%%%%%
\subsection{Acknowledgements}\

%%%%%%%%%%%%%%%%%%%%%%%
We thanks V. Gimeno his useful help.

%%%%%%%%%%%%%%%%%%%%%%
%%%%%%%%%%%%%%%%%%%%%%%%%%%%%%%%
\section{Preliminaries and comparison setting}\label{sec:preliminars}\
%%%%%%%%%%%%%%%%%%%%%%%%%%%%%%%%%
%%%%%%%%%%%%%%%%%%%%%%%%%%%%%

\bigskip

We are going to present some previous notions and results that will be instrumental in our work.

%%%%%%%%%%%%%%%%%%%%%
\subsection{Polar coordinates and the Laplacian on a Riemannian manifold}\
%%%%%%%%%%%%%%%%%%%%%%

\begin{definition}
Let us consider a complete Riemannian manifold $(M^n, g)$ and a point $o \in M$. Let us denote as $Cut(o)$ the cut locus of $o\in M$ and as $inj(o)=dist_M(o, Cut(o))$ the injectivity radius of the point $o \in M$. We shall denote too as $\SN_1 \subseteq \erre^n$ the unit sphere with center $\vec{0} \in \erre^n$.

We define, in the set $M \sim (Cut(o) \cup \{o\})$, the {\em polar coordinates} of any point $x \in M \sim (Cut(o) \cup \{o\})$ as the pair $(r(x), \ov{\theta}) \in (0, inj(o))\times \SN_1$, where $r(x):=r_o(x)= dist_M(o,x)$ is the distance from $o$ to $x$ realized by the shortest geodesic between these points which starts at $o$ with direction $\ov{\theta} \in \SN_1$. 
\end{definition}

The Riemannian metric $g$ in $M \sim (Cut(o) \cup \{o\})$ has in the polar coordinates the form
$$g=dr^2+\sum_{i,j=1}^{n-1}g_{i,j}(r,\ov{\theta}) d\theta^i d\theta^j\, ,$$
\noindent where $\ov{\theta}\equiv (\theta_1,...,\theta_{n-1}) \in \SN_1$ is a system of local coordinates in $\SN_1$ and  $g_{i,j}(r,\ov{\theta})=g\left(\left.\frac{\partial}{\partial \theta_i}\right\rvert_{(r,\ov{\theta})},\left.\frac{\partial}{\partial \theta_j}\right\rvert_{(r,\ov{\theta})}\right)$.

Thus, the matrix form of the metric $g$ in polar coordinates is a positive definite matrix given by
\begin{equation*}
	\mathfrak{G}=\left(\begin{array}{@{}c|c@{}}
		1 & \begin{matrix}
			0 & \cdots & 0
		\end{matrix} \\
		\hline
		\begin{matrix}
			0\\ \vdots\\ 0
		\end{matrix} &
		G
	\end{array}\right)\, ,
\end{equation*}
where $G$ is the matrix which elements are $g_{ij}$, \emph{i.e.}, $G=\left(g_{ij}\right)_{i,j\in\{1,\dots,n-1\}}$. And hence, we have, for any point $(r,\ov{\theta})\in M-\left(Cut(o)\cup\{o\}\right)$, that
\begin{equation*}
	\sqrt{\det\left(\mathfrak{G}(r,\ov{\theta})\right)}=\sqrt{\det\left(G(r,\ov{\theta})\right)}.
\end{equation*}

Then, (see for example \cite{Gri}, \cite{Cha1}), the Laplace operator of $M$ has the following expression in the polar coordinates
\begin{equation}\label{lapeq}
\Delta^M = \frac{\partial }{\partial r^2}+\frac{\partial}{\partial r}\big(log\DGr \big)\frac{\partial}{\partial r}+\Delta^{S^M_r(o)}\, ,
\end{equation}
\noindent where $\Delta^{\S^M_r(o)}$ is the Laplace operator in the geodesic sphere $S^M_r(o) \subseteq M$.

\begin{remark}
Along all the paper, given $o \in M$ and as long as $R <inj(o)$, we will use indistinctly the terms {\em  geodesic ball}, {\em geodesic sphere}, {\em metric ball}, {\em metric sphere}, {\em distance ball} and {\em distance sphere} to name the sets $B^M_R(o)$ and $S^M_R(o)$ respectively.
\end{remark}

Using this result we have the following

 \begin{proposition}\label{lapmean}
 Let $(M^n,g)$ be a complete Riemannian manifold and let $o\in M$ be a point of $M$. Then the normalized mean curvature vector field of the geodesic sphere $S^M_t(o)$, is given by
 $$\vec{H}_{S^M_t(o)}=-H_{S^M_t(o)}\nabla^M r\, ,$$
 \noindent where
 
 $$\hrm_{\srm_t^M}=\frac{1}{n-1}\Delta^M r(\gamma(t))=\frac{1}{n-1}\dfrac{\frac{\partial}{\partial t}\DG}{\DG}\,\,\forall t>0$$
 \noindent is the pointed inward mean curvature of $S^M_t(o)$ and $\gamma(t)$ is a unit geodesic starting at the point $o \in M$. 
 \end{proposition}
 \begin{proof}
 The proof is straightforward taking $\{\vec{e}_i(t)\}_{i=1}^n$ an orthonormal basis of $T_{\gamma(t)}\srm_t^M$, with $\vec{e}_n(t)=\nabla^Mr(\gamma(t))$, the unit normal to $\srm_t^M$ at $\gamma(t)$, pointed outward. Then,  after some computations,
 \begin{equation}
 \begin{aligned}
 \vec{H}_{S^M_t}&=\frac{1}{n-1} \big(\tr L_{\nabla^M r}\big) \nabla^M r=-\frac{1}{n-1} \Div^M(\nabla^Mr)\\&=-\frac{1}{n-1} \Delta^M r(\gamma(t))\nabla^Mr(\gamma(t))
 \end{aligned}
 \end{equation}
 \noindent so $$H_t=\langle \vec{H}_{S^M_t},-\nabla^Mr(\gamma(t))\rangle=\frac{1}{n-1}\Delta^Mr(\gamma(t)).$$
 \noindent The result follows now using equation (\ref{lapeq}).
 \end{proof}
 
 Given a domain (connected open set) $D$ in $M$, a function $u\in C^2(D)$ is \emph{harmonic} (resp. \emph{subharmonic}) if $\Delta^M u=0$ (resp. $\Delta^M u\geq 0$) on $D$. We gather the strong maximum principle and the Hopf boundary point lemma for subharmonic functions in the next statement.

\begin{theorem}
\label{th:mp}
Let $D$ be a smooth domain of a Riemannian manifold $M$. Consider a subharmonic function $u\in C^2(D)\cap C(\overline{D})$. Then, we have:
\begin{itemize}
\item[(i)] if $u$ achieves its maximum in $D$ then $u$ is constant,
\item[(ii)] if there is $p_0\in\partial D$ such that $u(p)<u(p_0)$ for any $p\in D$ then $\frac{\partial u}{\partial \nu}(p_0)>0$, where $\nu$ denotes the outer unit normal along $\partial D$.
\end{itemize}
\end{theorem}

\begin{proof}
The proof of (i) can be found in \cite[Cor.~8.15]{Gri2}. The proof of (ii) can be derived from (i) as in the  Euclidean case \cite[Lem.~3.4]{GT}.
\end{proof}

%%%%%%%%%%%%%%%%%%%%%%%
\subsection{Model Spaces}\label{subsec:Model}\
%%%%%%%%%%%%%%%%%%%%%%

\medskip

The model spaces $\MW$ are rotationally symmetric spaces defined as follows:

\begin{definition}\label{def:modelspaces}
	(See \cite{Gri},\cite{GreW}) A $\omega$-model $\MW$ is a smooth warped product with base $B^1=[0,R[\subset\R$ (on $0<R\leq\infty$), fiber $F^{n-1}=\SN_1$ (i.e., the unit $(n-1)$-sphere with standard metric), and warping function $\omega:\,[0,R[\rightarrow\R_+\cup\{0\}$ with $\omega(0)=0$, $\omega'(0)=1$, $\omega^{(2k)}(0)=0$ and $\omega(r)>0$ for all $k\in\N^*$ and for all $r>0$, where $\omega^{(2k)}$ denotes the even derivatives of the warping function.
	
	The point $o_\omega=\pi^{-1}(0)$, where $\pi$ denotes the natural projection onto $B^1$, is called \emph{center point} of the model space. If $R=+\infty$, then $o_\omega$ is a pole of $\MW$. We denote as $r=r(x)$ the distance to the pole $o_\omega$ of the point $x \in \MW$.
\end{definition} 

\begin{remark}\label{prop:SpaceForm}
	The simply connected space forms $M^n_{\omega_b}$ of constant sectional curvature $b$ can be constructed as $\omega$-models with any given point as center point using the warping functions
	
	\begin{equation}\label{eq:SpaceForm}
		\omega_b(r)=\begin{cases}
			\dfrac{1}{\sqrt{b}}\sin\left(\sqrt{b}\,r\right),& \quad\text{si }\,b>0,\\
			r, & \quad\text{si }\,b=0,\\
			\dfrac{1}{\sqrt{-b}}\sinh\left(\sqrt{-b}\,r\right), & \quad\text{si }\,b<0. 
		\end{cases}
	\end{equation}
	
	Note that for $b>0$, the warped metric $g_{\omega_b}=dr^2+\omega_b^2(r)g_{\SN_1}$ determined by the function $\omega_b(r)$ admits smooth extension to $r=\pi/\sqrt{b}$. For $b\leq 0$ any center point is a pole.
\end{remark}

In  \cite{O'N}, \cite{GreW}, \cite{Gri} and \cite{MP4}, we have a complete description of these model spaces, including the computation of their sectional curvatures $K_{o_\omega,\MW}$ in the radial directions from the center point $o_\omega$. They are determined by the radial function $K_{o_\omega,\MW}(\sigma_x)=K_\omega(r)=-\frac{\omega''(r)}{\omega(r)}$. Moreover, the normalized inward mean curvature of the distance sphere $S^\omega_r(o_\omega)$ of radius $r$ from the center point,  is, at the point $p=\gamma(r) \in  S^\omega_r(o_\omega)$, where $\gamma(t)$ is the normal geodesic parametrized by arclength joining $o_\omega$ and $p$

\begin{equation}\label{eq:WarpMean}
	H_{S^w_r}(p)=\eta_\omega(r)=\dfrac{\omega'(r)}{\omega(r)}=\dfrac{d}{dr}\ln\left(\omega(r)\right).
\end{equation}

In particular, in \cite{MP4} we introduce, for any given warping function $\omega(r)$, the \emph{isoperimetric quotient $q_\omega(r)$} for the corresponding $\omega$-model space $\MW$ as follows:

\begin{equation}\label{eq:Defq}
	q_\omega(r)=\dfrac{\Vol\left(\brm_r^\omega(o_\omega)\right)}{\Vol\left(\srm_r^\omega(o_\omega)\right)}=\dfrac{\int_0^r\omn(t)\,dt}{\omn(r)}.
\end{equation}

On the other hand, using  equation (\ref{lapeq}), the Laplace operator in $M^n_w$ is given by
\begin{equation}\label{laplacemodel}
\Delta^{M^n_w} = \frac{\partial }{\partial r^2}+(n-1)\frac{w'(r)}{w(r)}\frac{\partial}{\partial r}+\Delta^{S^{\omega}_r(o_\omega)}.
\end{equation}

Then, we have the following results concerning the mean exit time function  of the geodesic $R$-ball $\brm_r^\omega(o_\omega)\subseteq\MW$, (see \cite{MP4}):

\begin{proposition}\label{prop:MeanTorEqualities}
	Let $E_R^\omega$ the solution of the Poisson Problem \eqref{eq:moments1}, defined on the geodesic $R$-ball $\brm_R^\omega(o_\omega)$ in the model space $\MW$. 
	
	Then $E_R^\omega$ is a non-increasing radial function  given by
	
\begin{equation}\label{eq:EwR}
		E_R^\omega(x)=E_R^\omega(r_{o_\omega}(x))=\int_{r_{o_\omega}(x)}^R q_\omega(t)\,dt\, ,
\end{equation}
\noindent  where $r\equiv r_{o_\omega}(x)=dist_{M^n_{\omega}}(o_\omega, x)$ denotes the distance to the center point. Hence, it attains its maximum at $r=0$, with $E_R^\omega{'}(0)=0$ and $E_R^\omega{'}(r) < 0\,\,\,\forall r \in ]0,R]$ .

\end{proposition}
\begin{proof}
Using the expresion of the Laplace operator given in equation (\ref{laplacemodel}), it is straightforward to check that $E_R(r)=\int_r^R q_\omega(t)\,dt$ satisfies the 
equation
$$\Delta^{M^n_w} E_R=-1$$
\noindent  with boundary condition $E_R(R)=0$.
\end{proof}

%\begin{corollary}\label{prop:MeanTorEqualities}(see \cite{MP4}).
	%The torsional rigidity  $\CA_1\left(\brm_R^\omega\right)$ satisfies:	

	%\begin{equation}\label{eq:Apkenunabola}
		%\CA_1\left(\brm_R^\omega\right)=\int_{\brm_R^\omega}E_r^\omega\,d\widetilde{\sigma}=\vrm_0\int_0^R\omn(r)\left(\int_r^R q_\omega(t)\,dt\right)\,dr
	%\end{equation}

%\noindent where $\vrm_0$ is the volume of the unit sphere $\SN_1$. Differentiating with respect to $R$ gives
	
%\begin{equation}\label{eq:derivative2}
	%	\dfrac{d}{dR}\CA_1\left(\brm_R^\omega\right)=q^2_\omega(R)\Vol\left(\srm_r^\omega\right)
%\end{equation}
%\noindent and an integration of the latter equality, gives us the following alternative expression of the torsional rigidity:
%\begin{equation}\label{eq:Apk2}
		%\CA_1\left(\brm_R^\omega\right)=\int_{\brm_R^\omega}q^2_\omega\,d\widetilde{\sigma}
	%\end{equation} 
%\end{corollary}

%%%%%%%%%%%%%%%%%%%%%%%%%%%%%%%%%%%%%%%%
\subsection{Balance conditions}\
%%%%%%%%%%%%%%%%%%%%%%%%%%%%%%%%%%%%%%%%%

We present now a purely intrinsic condition on the general model spaces $\MW$, (see  \cite{MP4}), which will play a key role in the last section of the paper:

\begin{definition}
	A given $\omega$-model space $\MW$ is  \emph{balanced from above} if we have the inequality
	
	\begin{equation}\label{eq:balancedabove1}
		q_\omega(r)\eta_\omega(r)\leq\dfrac{1}{n-1},\quad\text{for all } r\geq 0.
	\end{equation}
	\end{definition}

In \cite{MP4} it was proved the following characterization of the balance condition defined previously:

\begin{proposition}\label{prop:balancedequivalences}
Let us consider the $\omega$-model space $\MW$. Then $\MW$ is \emph{balanced from above} if and only if the following equivalent conditions hold:
		
		\begin{align}
			\dfrac{d}{dr}\left(q_\omega(r)\right) &  \geq 0,\label{eq:balancedabovediff}\\
			\omega^n(r) & \geq (n-1)\,\omega'(r)\int_0^r\omn(t)\,dt.\label{eq:balancedabove2}
		\end{align}
	
\end{proposition}

Also in \cite{MP4} it were listed several examples of balanced from above $\omega$-model spaces $\MW$. Here we enumerate some of them:

\begin{examples}\

	\begin{enumerate}
	\item Every $\omega_b$-model space $M^n_{\omega_b}=[0,R[\times_{\omega_b}\SN_1$ of constant positive sectional curvature $b>0$ and $R <\frac{\pi}{2\sqrt{b}}$ is  balanced from above. In fact, when $b >0$ and for $r>0$ we have that  (\ref{eq:balancedabove2}) is a strict inequality, (see Lemma 2.4 in \cite{MP3}).
	\item On the other hand, the $\omega_b$-model spaces $M^n_{\omega_b}$ of constant non-positive sectional curvature $b\leq 0$ are  balanced from above too. In fact, when $b <0$, we have that inequality (\ref{eq:balancedabove2}) is equivalent to inequality 
	$$\int_0^r \sinh^{n-1}(\sqrt{-b} t) \leq \frac{\sinh^n(\sqrt{-b} r)}{\sqrt{-b}(n-1)\cosh(\sqrt{-b}r)}$$
	which holds for all $r>0$ because $\tanh^2(\sqrt{-b} r) \leq 1\,\,\forall r>0$. The case $b=0$ is trivial.
	\item Let us consider the $\omega$-model space $M^n_\omega$, being $\omega(t):=t+t^3, t\in [0, \infty)$. This model space is balanced from above.

		\end{enumerate}
\end{examples}

%%%%%%%%%%%%%%%%%%%%%%%%%%
\subsection{Symmetrization into Model Spaces}\label{subsec:Symm}\
%%%%%%%%%%%%%%%%%%%%%%%%%%%%%%%%%
%%%%%%%%%%%%%%%%%%%%%%%%%%%%%%%%%%

\medskip

As in \cite{MP4} we use the concept of \emph{Schwarz-symmetrization} as considered in e.g., \cite{Ba}, \cite{Po}, or, more recently, in \cite{Mc} and \cite{Cha2}. For the sake of completeness, we review and show some facts about this instrumental concept, in the context of Riemannian manifolds.

\begin{definition}\label{def:Symm}
	Let $(M^n, g)$ be a complete Riemannian manifold. Suppose that $\drm \subseteq M$ is a precompact open connected domain in $M^n$. Let $(\MW, g_\omega)$ be a rotationally symmetric model space, with pole $o_\omega \in \MW$. Then the \emph{$\omega$-model space symmetrization of \, $\drm$} is denoted by $\drm^{*_{\omega}}$ and is defined to be the unique $L(D)$-ball in $\MW$, centered at $o_\omega$
	$$D^{*_{\omega}}:=B^\omega_{L(D)}(o_\omega)$$
	\noindent satisfying

	\begin{equation*}
		\Vol(D)=\Vol\left(B^\omega_{L(D)}(o_\omega)\right).
	\end{equation*}
	
	In the particular case that $D$ is a geodesic $R$-ball $\brm_R^M(o)$ in $M$ centered at $o \in M$, then the radius $L(\brm_R^M(o))$ is some increasing function  $s(R)=L(\brm_R^M(o))$ which depends on the geometry of $M$, so we can write
	
	\begin{equation*}
		\brm_R^M(o)^{*_{\omega}}=B^\omega_{s(R)}(o_\omega)
	\end{equation*}

	 \noindent and this symmetrization $B^\omega_{s(R)}(o_\omega)$ satisfies
	 	\begin{equation}\label{symball}
		\Vol(\brm_R^M(o))=\Vol(B^\omega_{s(R)}(o_\omega)).
	\end{equation}
	\end{definition}
	\begin{remark}
	When it is clear from the context, we write  $\drm^{*}$  instead of $\drm^{*_{\omega}}$.
	
	Along the rest of the paper, and if there is not confusion, we shall omit the centers $o \in M$ and  $o_\omega \in M^n_\omega$ when we refer to the balls $\brm_{r}^M(o)$ and $\brm_{r}^\omega(o_\omega)$ and the spheres $\srm_{r}^M(o)$ and $\srm_{r}^\omega(o_\omega)$.
	\end{remark}

	%Note, on the other hand, that the center $o_\omega$ of the rotationally symmetric space used as a model, $(\MW, g_\omega)$, must be a pole, (hence, $inj(o_\omega)=\infty$), and moreover, the warping function $\omega(r)$ must satisfy that $\lim_{R \to \infty} \omega(R)=\infty$, (and hence, $\lim_{R \to \infty} \Vol(B^\omega_{s(R)}(o_\omega))=\infty$), in order to \lq\lq catch" into the definition of symmetrization domains $D$ with arbitrary volumes.
	%\end{remark}
%\begin{remark}
%Let us note that, by virtue of equation (\ref{symball}) and equations (\ref{eq:IgualtatsVolumsambDerivades}), we can express $s(R)$ in terms of the volume of spheres in the manifold $M$ as
%\begin{equation}
%\Vol(S^M_R(o))=\Vol(\SN_1)\omega^{n-1}(s(R)) s'(R)
%\end{equation}
%\end{remark}
\medskip

Given  $f: D \rightarrow \erre^+$ a smooth non-negative function on $\drm$, we are going to introduce the notion of  {\em $\omega$- symmetrization} $f^{*_{\omega}}: \drm^{*_{\omega}}\rightarrow \erre^+$. But first, we will show some useful facts.

\begin{definition}\label{def:simfuncprevi}
	Let $(M^n, g)$ be a complete Riemannian manifold, $D \subseteq M$ a precompact domain in $M$ and $f:\drm\subseteq M\longrightarrow\R^+$ a smooth non-negative function on $\drm$. For $t\geq 0$ we define the sets
	\begin{equation*}
		\drm(t):=\{x\in\drm\,|\,f(x)\geq t\}\subseteq M
	\end{equation*}
	\noindent and
	\begin{equation*}
		\Gamma(t):=\{x\in\drm\,|\,f(x)=t\}.
	\end{equation*}
\end{definition}

\begin{remark}\

\begin{enumerate}
\item The set $D(t)$ is precompact for all $t\geq 0$  and moreover, $\partial\drm(t)=\Gamma(t) \subseteq D(t)$. 
\item Note too that $D(0)=D$ and that if $t_1\leq t_2$ then $\drm(t_2)\subseteq\drm(t_1)$. 
\item  If $T:= \sup_{x \in D} f(x)$, then $D(t)= \emptyset \,\,\forall t > T$, and hence, $\Vol(D(t))=0\,\,\forall t \geq T$. 
\item Therefore, we have a family of nested sets $\{D(t)\}_{t \in [0,T]}$ that covers $D$.

\end{enumerate}
\end{remark}

Now, we define the {\em symmetrization} of a function:

\begin{definition}\label{def:simfunc}
	Let $(M^n, g)$ be a complete Riemannian manifold, $D \subseteq M$ a precompact domain in $M$ and $f:\drm\subseteq M\longrightarrow\R^+$ a smooth non-negative function on $\drm$. Let $(\MW, g_\omega)$ be a rotationally symmetric model space.	Then the \emph{$\omega$-symmetrization of $f$} is the function $f^{*_{\omega}}:\drm^{*_{\omega}}\longrightarrow\R$ defined, for all $x^* \in  \drm^{*_{\omega}}$, by
	
	\begin{equation*}
		f^{*_{\omega}}(x^*)=\sup\{t\geq 0\,|\, x^*\in \drm(t)^{*_\omega}\}.
	\end{equation*}
	Note that the symmetrization $f^{*_{\omega}}$ ranges on $[0,T]$, namely, $f^{*_{\omega}}: D^{*_{\omega}} \rightarrow [0,T]$, where $T:= \sup_{x \in D} f(x)$.
	%Note that $\drm(t)^*=\brm_{R(t)}^\omega$, the balls defined in Remark \ref{remark:Symmdetails}.
\end{definition}
\begin{remark}\
\begin{enumerate}
	\item When it is clear from the context, we write  $f^{*}$  instead of $f^{*_{\omega}}$ and $D^*$ instead of $D^{*_{\omega}}$.
	%\item Given the precompact domain $D \subseteq M$  and $f:\drm\subseteq M\longrightarrow\R^+$ a non-negative function on $\drm$,  the $\omega$- symmetrization of $\drm(t)$, is the geodesic ball $\drm(t)^*=\brm_{R(t)}^\omega(o_\omega)$ such that 
	%\begin{equation*}\label{symmt}
	%\Vol\left(\drm(t)\right)=\Vol\left(\drm(t)^*\right)=\Vol\left(\brm_{R(t)}^\omega(o_\omega)\right)
	%\end{equation*}
	%\item Note also that, as $D(0)=D$, then $D(0)^*=D^*$, i.e., $R(0)=L(D)$, the radius defined in Definition \ref{def:Symm}. On the other hand, if $T:= \sup_{x \in D} f(x)$, then $D(t)= \emptyset \,\,\forall t > T$, and hence, $\Vol(D(t))=0\,\,\forall t \geq T$. Then $R(t)=0\,\,\forall t \geq T$
	\item By Sard's theorem, if $D_f \subseteq D$ denotes the set of critical points of $f$, the set $S_f =f(D_f)\subseteq [0,T]$ of critical values of $f$ has null measure, and the set of regular values of $f$, $R_f=[0,T]\sim S_f$ is open and dense in $[0,T]$. In particular, for any $t \in R_f$, the set $\Gamma(t)=\{x\in\drm\,|\,f(x)=t\}$ is a smooth embedded hypersurface in $D$ and $\Vert \nabla^M f\Vert$ does not vanish along $\Gamma(t)$.
	\end{enumerate}
	\end{remark}
	
	With these observations in hand, we have the following
	
	\begin{definition}\label{defR}
	Let $(M^n, g)$ be a complete Riemannian manifold and let $(\MW, g_\omega)$ be a rotationally symmetric model space. Given the precompact domain $D \subseteq M$  and $f: D \subseteq M\longrightarrow\R^+$ a smooth non-negative function on $D$, let us define the function
	$$\widetilde{r}: [0, T] \rightarrow [0, L(D)]$$
	
	\noindent such that, for all $t \in [0, T]$, $\widetilde{r}(t)$ is defined as the radius of the symmetrization 
	$$D(t)^*=\brm_{\widetilde{r}(t)}^\omega(o_\omega)$$
	\noindent satisfying
	\begin{equation*}
	\Vol\left(\drm(t)\right)=\Vol\left(\brm_{\widetilde{r}(t)}^\omega(o_\omega)\right).
	\end{equation*}
	\end{definition}
	
	\begin{remark}
	Note  that, as $D(0)=D$, then $D(0)^*=D^*$, i.e., $\widetilde{r}(0)=L(D)$, the radius defined in Definition \ref{def:Symm}, and $D^*=\brm_{\widetilde{r}(0)}^\omega(o_\omega)$. On the other hand, as  $\Vol(D(t))=0\,\,\forall t \geq T$, then $\widetilde{r}(t)=0\,\,\forall t \geq T$.
	\end{remark}
	
 Concerning this last definition, we have the following result, which will play an important r\^ole in the proof of Proposition \ref{prop:ineqpsiE}:

\begin{lemma}\label{prop:Rfuncsymm}
	The function $\widetilde{r}: [0, T] \rightarrow [0, L(D)]$  is non-increasing. In particular, for all regular values $t \in R_f$, the function $\widetilde{r}\vert_{R_f}: R_f \subseteq [0, T] \rightarrow [0, L(D)]$ satisfies
	\begin{equation*}
		\widetilde{r}'(t)=-\dfrac{\int_{\partial D(t)}\norm{\nabla^M f}^{-1}\,d\mu_t}{\Vol\left(\srm_{\widetilde{r}(t)}^\omega\right)} < 0
	\end{equation*}
	
	\noindent so it  is strictly decreasing in $R_f$, and hence, injective (and bijective onto its image).
\end{lemma}
\begin{remark}
Note that when $R_f=[0,T]$, then $\widetilde{r}: [0, T] \rightarrow [0, L(D)]$ is bijective.
\end{remark}

\begin{proof}

When  $t_1\leq t_2$, then $\drm(t_2)\subseteq\drm(t_1)$ and hence $\Vol(\drm(t_2)) \leq \Vol(\drm(t_1))$, so $\Vol(\brm_{\widetilde{r}(t_2)}^\omega(o_\omega)) \leq \Vol(\brm_{\widetilde{r}(t_1)}^\omega(o_\omega))$ and hence $\widetilde{r}(t_2) \leq \widetilde{r}(t_1)$.

	On the other hand, given $t \in R_f$, let us denote as:
	
	\begin{equation*}
\vrm(t)  =\Vol\left(\drm(t)\right)=\Vol\left(\brm_{\widetilde{r}(t)}^\omega\right).
	\end{equation*}
	
	Then, \begin{equation*}
		\vrm'(t)=\Vol\left(\srm_{\widetilde{r}(t)}^\omega\right)\widetilde{r}'(t)
	\end{equation*}
	
	\noindent and as $\partial\drm(t)=\Gamma(t)=\{x\in\drm\,|\,f(x)=t\}$, by the co-area formula (see \cite{Cha1}, \cite{Sa}), and as $t \in R_f$, we have 
	
	%\begin{equation*}
		%\vrm'(t)=\Vol\left(\srm_{R(t)}^\omega\right)R'(t)=-\int_{\Gamma(t)}\norm{\nabla^M f}^{-1}\,d\mu_t < 0
	%\end{equation*}

	%\noindent and hence
	
	\begin{equation*}
		\widetilde{r}'(t)=-\dfrac{\int_{\partial D(t)}\norm{\nabla^M f}^{-1}\,d\mu_t}{\Vol\left(\srm_{\widetilde{r}(t)}^\omega\right)} < 0
	\end{equation*}
	
	\noindent for all $t \in R_f$. Therefore, $\widetilde{r}\vert_{R_f}$ is strictly decreasing.
\end{proof}

To finish this subsection, we are going to prove in Theorem \ref{prop:Propietatssymmobjects} that, given $f:\drm\subseteq M\longrightarrow\R^+$ a non-negative function defined on the precompact domain $D$, the symmetrized function $f^*:D^{*_{\omega}}\longrightarrow\R$ is a radial function, and that $f$ and $f^*$ are both equimeasurable.

\begin{theorem}\label{prop:Propietatssymmobjects}
	Let $(M^n, g)$ be a complete Riemannian manifold, $D \subseteq M$ a precompact domain in $M$ and $f:\drm\subseteq M\longrightarrow\R^+$ a non-negative and smooth function on $\drm$. Let $(\MW, g_\omega)$ a rotationally symmetric model space such that its center $o_\omega$ is a pole. The symmetrized objects $f^*$ and $\drm^*$ satisfy the following properties:
	
	\begin{enumerate}
		\item \label{prop:Propietatssymmobjects1} The function $f^*$ depends only on the geodesic distance to the center $o_\omega$ of the ball $\drm^*$ in $\MW$ and is non-increasing.
		
		\item \label{prop:Propietatssymmobjects2} The functions $f$ and $f^*$ are equimeasurable in the sense that 
		
		\begin{equation}\label{eq:eqVol}
			\Vol_M\left(\{x\in\drm\,|\,f(x)\geq t\}\right)=\Vol_{\MW}\left(\{x^*\in\drm^*\,|\,f^*(x^*)\geq t\}\right)
		\end{equation}
		
		for all $t\geq 0$.
	\end{enumerate}
\end{theorem}

\begin{proof}

We are going to prove first statement. Let us consider $x_1^*,x_2^*\in\drm^*=\brm_{\widetilde{r}(0)}^\omega(o_\omega)$ such that $r_{o_\omega}(x_1^*)=r_{o_\omega}(x_2^*)$. Then  it is evident that $x_1^* \in \brm_{\widetilde{r}(t)}^\omega(o_\omega)$ if and only if $x_2^* \in \brm_{\widetilde{r}(t)}^\omega(o_\omega)$ for all $t \in [0,T]$. Hence, 

	\begin{equation*}
	%\begin{aligned}
	f^*(x_1^*)=\sup\left\{t\geq 0\,|\,x_1^*\in \brm_{\widetilde{r}(t)}^\omega(o_\omega)\right\}
	= \sup\left\{t\geq 0\,|\,x_2^*\in \brm_{\widetilde{r}(t)}^\omega(o_\omega)\right\}=f^*(x_2^*)
	%\end{aligned}
	\end{equation*}
	\noindent which means that $f^*$ is a radial function. Namely, $f^*$ depends only on the geodesic distance to the center $o_\omega$, $f^*(x^*)=f^*(r_{o_\omega}(x^*))$.
	
	To see that $f^*$ is non-increasing, let us consider $x_1^*,x_2^*\in\drm^*$ such that $r_{o_\omega}(x_1^*)\leq r_{o_\omega}(x_2^*)$. We are going to see that $t_1:=f^*(x_1^*) \geq t_2:=f^*(x_2^*)$. 
	
	As $$f^*(x_2^*)= \sup\left\{t\geq 0\,|\,x_2^*\in \brm_{\widetilde{r}(t)}^\omega(o_\omega)\right\}=\sup\left\{t\geq 0\,|\,r_{o_\omega}(x_2^*) \leq \widetilde{r}(t)\right\}=t_2$$
	\noindent then, if $t \leq t_2$, we have that $x_2^*\in \brm_{\widetilde{r}(t)}^\omega(o_\omega)$, so $r_{o_\omega}(x_2^*) \leq \widetilde{r}(t)\,\,\forall t \leq t_2$. In particular, $r_{o_\omega}(x_1^*)\leq r_{o_\omega}(x_2^*) \leq \widetilde{r}(t_2)$, so $x_1^* \in \brm_{\widetilde{r}(t_2)}^\omega(o_\omega)$ and therefore, $t_1=f^*(x_1^*)=\sup\left\{t\geq 0\,|\,x_1^*\in \brm_{\widetilde{r}(t)}^\omega(o_\omega)\right\} \geq t_2=f^*(x_2^*)$.
%By definition, $x_2^* \in \brm_{R(t)}^\omega(o_\omega)$ for all $t \leq t_2=f^*(x_2^*)$. Hence, $r_{o_\omega}(x_1^*)\leq r_{o_\omega}(x_2^*) \leq R(t_2)$, so $x_1^* \in \brm_{R(t_2)}\omega(o_\omega)$ and therefore, $t_1=f^*(x_1^*)=\sup\left\{t\geq 0\,|\,x_1^*\in \brm_{R(t)}^\omega(o_\omega)\right\} \geq t_2=f^*(x_2^*)$.
	\medskip
	
	To prove second statement, note that, for all $t>0$, we have, by Definitions \ref{def:simfunc} and \ref{defR}, 
	
	\begin{equation*}
		\begin{split}
			\drm(t)^* & =\brm_{\widetilde{r}(t)}^\omega(o_\omega)=\left\{x^*\in\drm^*\,|\,f^*(x^*)\geq t\right\}.
		\end{split}
	\end{equation*}
	
	\noindent In fact, if $x^* \in \brm_{\widetilde{r}(t)}^\omega(o_\omega)$, then $f^*(x^*)=\sup\left\{t\geq 0\,|\,x^*\in \brm_{\widetilde{r}(t)}^\omega(o_\omega)\right\} \geq t$ and, conversely, if $f^*(x^*)=\sup\left\{t\geq 0\,|\,x^*\in \brm_{\widetilde{r}(t)}^\omega(o_\omega)\right\} \geq t$, then $x^* \in \brm_{\widetilde{r}(t)}^\omega(o_\omega)$.

	Therefore, since  $\drm(t)=\left\{x\in\drm\,|\,f(x)\geq t\right\}$, we obtain that
	
	\begin{equation*}
		\Vol\left(\left\{x\in\drm\,|\,f(x)\geq t\right\}\right)=\Vol\left(\drm(t)\right)=\Vol\left(\drm(t)^*\right)=\Vol\left(\left\{x^*\in\drm^*\,|\,f^*(x^*)\geq t\right\}\right).
	\end{equation*}
\end{proof}

%%%%%%%%%%%%%%%%%%%%%%%%%%%%%%%%
%%%%%%%%%%%%%%%%%%%%%%%%%%%%%%%%%%
\section{Mean Exit Time Comparison}\label{sec:MeanComp}\
%%%%%%%%%%%%%%%%%%%%%%%%%%%%%%%%
%%%%%%%%%%%%%%%%%%%%%%%%%%%%%%%

\bigskip
 
 We start this section with the notion of {\em transplanted mean exit time}.
 
 \begin{definition}\label{transplanted}
 Let $(M,g)$ a complete Riemannian manifold and $(\MW,g_{\omega})$ a model space with center $o_\omega$. Given $o \in M$, let us consider a geodesic $R$-ball $B^M_R(o)$, with $0 <R<\, inj(o)$ and the geodesic $R$-ball in $\MW$, centered at the center $o_\omega$, $B^{\omega}_R(o_\omega)$. Let $E_R^M$ and $E_R^\omega$ be the mean exit time functions defined on $B^M_R(o)$ and $B^{\omega}_R(o_\omega)$, respectively. 
 
 Now, we {\em transplant} the radial mean exit time function of $\MW$ to $M$ by defining the function $\E_R^\omega:\, \brm_R^M \rightarrow \erre$ as $\E_R^\omega(x):=E_R^\omega\left(r_o(x)\right)\,\,\forall x \in \brm_R^M$ where $r_o$ is the distance function to $o$, the center of the ball $\brm_R^M(o)$. 
 
 The function $\E_R^\omega$ is a radial function called the {\em transplanted mean exit time} function in $\brm_R^M$.
 \end{definition}

We can compare the transplanted mean exit time function  $\E_R^\omega$ defined in a geodesic ball $B^M_R$ with the mean exit time function $E^M_R$ corresponding with this ball. Remember that, along the text,   we can omit the centers $o \in M$ and  $o_\omega \in M^n_\omega$ when we refer to the balls $\brm_{r}^M(o)$ and $\brm_{r}^\omega(o_\omega)$ and the spheres $\srm_{r}^M(o)$ and $\srm_{r}^\omega(o_\omega)$.

Our first result in this regard is following:

\begin{proposition}\label{prop3.2}
Let $(M^n,g)$ be a complete Riemannian manifold and let $(\MW,g_\omega)$ be a rotationally symmetric model space with center $o_\omega \in \MW$. Let $o \in M$ be a point in $M$ and let us suppose that  $inj(o) \leq inj(o_\omega)$. Let us consider a geodesic ball $B^M_R(o)$, with $R < inj(o) \leq inj(o_w)$. Then the following assertions are equivalent:
\begin{enumerate}
%\item There exists $p \in B^M_R(o)$such that $E^M_R(p)=\E_R^\omega(p)$
\item $E^M_R= \E_R^\omega \,\,\text{on}\,\,\,B^M_R(o)$.
\item $H_{S^w_r(o_\omega)}=H_{S^M_r(o)}\,\,\,\forall r \in ]0,R]$.
\end{enumerate}
\noindent where $\hrm_{\srm_r^M(o)}$ denotes the mean curvature of the geodesic $r$- sphere $\srm_r^M(o) \subseteq M$ and $\hrm_{\srm_r^\omega(o_\omega)}$ is the corresponding mean curvature of the geodesic $r$-sphere $\srm_r^\omega(o_\omega) \subseteq \MW$. 
\end{proposition}

\begin{proof}
Using polar coordinates $(r,\ov{\theta})$ in $M- (Cut(o)\cup\{o\})$, equations (\ref{lapeq}) and (\ref{laplacemodel}) and applying Maximum Principle, equality $E^M_R= \E_R^\omega \,\,\text{on}\,\,\,B^M_R(o)$ is equivalent to equality 
$$\Delta^M\E_R^\omega(r, \ov{\theta})=\Delta^M E^M_R(r, \ov{\theta})=-1=\Delta^{M^n_\omega}E^\omega_R\,\,\,\,\forall (r, \ov{\theta})$$ 
\noindent which, in its turn, applying Proposition \ref{lapmean} and equations (\ref{eq:WarpMean}) and (\ref{laplacemodel}), is equivalent to equality, for all $(r,\ov{\theta}) \in [0,R]\times \SN_1$:
$$\E_R^\omega{''}(r)+(n-1) H_{S^M_r(o)}\,\E_R^\omega{'}(r)=E_R^\omega{''}(r)+(n-1)H_{S^w_r(o_\omega)}\,E_R^\omega{'}(r)$$
and, as for all $r \in ]0,R]$, $\E_R^\omega{''}(r)= E_R^\omega{''}(r)$ and $\E_R^\omega{'}(r)=E_R^\omega{'}(r)<0\,\forall r \in ]0,R]$, this last equality is equivalent to equality
$$H_{S^M_r(o)}=H_{S^w_r(o_\omega)}\,\,\forall r \in ]0,R]$$.
\end{proof}

Now, we can state the following  comparison theorem: 

\begin{theorem}\label{teo:MeanComp}
	Let $(M^n,g)$ be a complete Riemannian manifold and let $(\MW,g_\omega)$ be a rotationally symmetric model space with center $o_\omega \in \MW$. Let $o \in M$ be a point in $M$ and let us suppose that  $inj(o) \leq inj(o_\omega)$. Let us consider a metric ball $B^M_R(o)$, with $R < inj(o) \leq inj(o_w)$. Let us suppose moreover that
	
	\begin{equation}\label{eq:meancurvatureconditions3}
		\hrm_{\srm_r^\omega(o_\omega)}\leq\,(\geq)\, \hrm_{\srm_r^M(o)}\quad\text{for all}\quad 0 < r \leq R
	\end{equation}
	
	\noindent where $\hrm_{\srm_r^M}$ denotes the mean curvature of the metric $r$- sphere $\srm_r^M(o) \subseteq M$ and $\hrm_{\srm_r^\omega}$ is the corresponding mean curvature of the metric $r$-sphere $\srm_r^\omega(o_\omega) \subseteq \MW$. 
	
	Then, we have the inequality
	\begin{equation}\label{eq:MeanComp}
		\E_R^\omega\geq\,(\leq)\, E_R^M\quad\text{in}\quad\brm_R^M (o)\, ,
	\end{equation}
	
	\noindent where $\E_R^\omega(x):=E_R^\omega\left(r_o(x)\right)$ is the {\em transplanted mean exit time} function in $\brm_R^M(o)$.
		
 Moreover, if there exists $p \in B^M_R(o)$ such that $\E_R^\omega(p)=E^M_R(p)$, then 
$$\E_R^\omega=E_R^M\,\quad\text{in}\,\,\,\brm_R^M(o)$$
\noindent and hence,
$$H_{S^w_r(o_\omega)}=H_{S^M_r(o)}\,\,\,\forall r \in ]0,R].$$

\end{theorem}

\begin{proof}
	To prove first assertion, let us consider polar coordinates $(r, \ov{\theta}) \in [0, inj(o))\times \SN_1$ centered at the center $o\in M$ of the geodesic ball $B^M_R$, with $R < inj(o)$, (as before and along the rest of the paper, we shall omit the center point of the ball $o$ if there is not confusion). By definition of  $\E_R^\omega$, and using equation (\ref{eq:EwR}) we have that this radial function satisfies

	%As $\E_R^\omega$ is a radial function, (which means that $\E_R^\omega(r, \ov{\theta})=\E_R^\omega(r)\,\,\forall (r,\ov{\theta})\in ]0,R]\times \SN_1$), we have that
	
	\begin{equation}\label{eq:MeanTransDecreixent}
		\E_R^\omega{'}(r)=E_R^\omega{'}(r)< 0, \quad\text{for all } r\in ]0,R],
	\end{equation}
	
	\noindent and, since  $\Delta^{\MW}E_R^\omega=-1\,\text{on} \, B^\omega_R$,
	
	%Furthermore, since $\E_R^\omega$ is radial, we have the following expressions of the Laplace-Betrami operator, using polar coordinates centered at the center of the geodesic ball $B^M_R$, and equation (\ref{lapeq}):
	
	%\begin{equation}\label{eq:MeaniMeanTransLap}
		%\begin{split}
			%\Delta^M\E_R^\omega(r, \ov{\theta}) & = \E_R^\omega{''}(r)+\dfrac{\frac{\partial}{\partial r}\DGr}{\DGr}\,\E_R^\omega{'}(r)\\
			%\Delta^{\MW}E_R^\omega(r) & =E_R^\omega{''}(r)+(n-1)\,\dfrac{\omega'(r)}{\omega(r)}E_R^\omega{'}(r)
		%\end{split}
	%\end{equation}

	\begin{equation*}
		\E_R^\omega{''}(r)=E_R^\omega{''}(r)=-1-(n-1)\,\dfrac{\omega'(r)}{\omega(r)}\,E_R^\omega{'}(r).
	\end{equation*}
	
	\noindent Therefore, using equation (\ref{lapeq}) and  applying Proposition \ref{lapmean} and equations (\ref{eq:WarpMean}) and (\ref{laplacemodel}), we have, for all $(r,\ov{\theta}) \in ]0,R]\times \SN_1$:
	
	\begin{equation}\label{eq:MeanTransLapSubstituida}
		\Delta^M\E_R^\omega(r,\ov{\theta})=-1+(n-1)\left(\hrm_{\srm_r^M}-\hrm_{\srm_r^\omega}\right)E_R^\omega{'}(r)\, .
	\end{equation}

	Then, from equations \eqref{eq:MeanTransLapSubstituida} and \eqref{eq:MeanTransDecreixent}, and assuming inequality $\hrm_{\srm_r^\omega}\,\leq\, \hrm_{\srm_r^M}\,\,\text{for all}\, r>0$ we obtain that
	
	\begin{equation}\label{eq:MeaniMeanTransDesigLap}
		\Delta^M\E_R^\omega(r,\ov{\theta})\leq -1=\Delta^M E_R^M(r,\ov{\theta}),\quad\text{for all }(r,\ov{\theta})\in]0,R]\,\times\,\SN_1\, .
	\end{equation}
	
	Thus
	
	\begin{equation*}
		\Delta^M\left(E_R^M-\E_R^\omega\right)(r,\ov{\theta})\geq 0\,\,\text{on}\,\, B^M_R
	\end{equation*}
\noindent and since $\left(E_R^M-\E_R^\omega\right)(R)=0$ we have, applying the strong maximum principle  	
	\begin{equation*}
		\E_R^\omega \geq E_R^M\,\,\text{ on}\,\, B^M_R
	\end{equation*}
	
	\noindent  as we wanted to prove. 
	We obtain opposite inequalities with same arguments, assuming that $\hrm_{\srm_r^\omega}\geq\, \hrm_{\srm_r^M}\quad\text{for all}\quad r>0$.
	\medskip
	
	We are going to prove second assertion assuming that 
	$$\hrm_{\srm_r^\omega(o_\omega)}\leq\, \hrm_{\srm_r^M(o)}\quad\text{for all}\quad 0 < r \leq R\, .$$
	
	\noindent For that, let us suppose that there exists $p \in B^M_R$ such that $\E_R^\omega(p)=E^M_R(p)$. Therefore, we have that $\Delta^M\left(E_R^M-\E_R^\omega\right) \geq 0$ on $B^M_R$ and that $E^M_R - \E_R^\omega \leq 0=(E_R^M-\E_R^\omega)(p)$ on $B^M_R$. Hence, $E_R^M-\E_R^\omega$ attains its maximum in $B^M_R$. Applying the strong maximum principle, the diference function $E_R^M-\E_R^\omega =C$ is constant on $B^M_R$ and, by continuity, as $E_R^M-\E_R^\omega =0$ on $\partial B^M_R=S^M_R$, then $C=0$.	
		
\end{proof}

\begin{corollary}\label{cor:MeanComp}

Let $(M^n,g)$ be a complete Riemannian manifold and let $(\MW,g_\omega)$ be a rotationally symmetric model space with center $o_\omega \in \MW$. Let $o \in M$ be a point in $M$ and let us suppose that  $inj(o) \leq inj(o_\omega)$. Let us consider a metric ball $B^M_R(o)$, with $R < inj(o) \leq inj(o_w)$. Let us suppose moreover that
	
	\begin{equation}\label{eq:meancurvatureconditions4}
		\hrm_{\srm_r^\omega}\leq\,(\geq)\, \hrm_{\srm_r^M}\quad\text{for all}\quad 0 < r \leq R.
	\end{equation}
	
	\noindent Then we have the isoperimetric inequalities
		
		\begin{equation} \label{cor:MeanComp1}
			\dfrac{\Vol\left(\brm_r^\omega(o_\omega)\right)}{\Vol\left(\srm_r^\omega(o_\omega)\right)}\geq\,(\leq)\,\dfrac{\Vol\left(\brm_r^M(o)\right)}{\Vol\left(\srm_r^M(o)\right)}\quad\text{for all}\quad 0<r \leq R.
		\end{equation}
		
		\noindent Moreover, equality in inequalites (\ref{cor:MeanComp1}) for some radius $r_0 \in ]0,R]$ implies that $$H_{S^w_r(o_\omega)}=H_{S^M_r(o)}\,\,\,\forall r \in ]0,r_0].$$
		
		 \noindent As a consequence of inequalities (\ref{cor:MeanComp1}), for all $0 < r \leq R$, we have
		
		\begin{equation}\label{ineqVolumes}
		\begin{aligned}
			\Vol\left(\brm_r^\omega(o_\omega)\right)\leq\,&(\geq)\,\Vol\left(\brm_r^M(o)\right),\\
			\Vol(S^\omega_r(o_\omega)) \leq\,&(\geq)\, \Vol(S^M_r(o)).
			\end{aligned}
		\end{equation}
		\noindent Finally,  equality 
		$$\Vol\left(\brm_{r_0}^\omega(o_\omega)\right)=\Vol\left(\brm_{r_0}^M(o)\right)$$
		\noindent  for some radius $r_0 \in ]0,R]$ implies that $$H_{S^w_r(o_\omega)}=H_{S^M_r(o)}\,\,\,\forall r \in ]0,r_0].$$
\end{corollary}

\begin{proof}

Let us fix one radius $r \in ]0,R]$. The proof follows the lines of the proof of Theorem 1.1 and Corollary 1.2 in \cite{Pa2}, adapting it to this intrinsic context and using the new hypotheses.

First, let us assume that $\hrm_{\srm_r^\omega}\leq \hrm_{\srm_r^M}$, for all $ 0<r \leq R$. If we fix $r \in ]0,R]$, then we have, in particular, that $\hrm_{\srm_s^\omega}\leq \hrm_{\srm_s^M}$, for all $ 0<s \leq r$. We can apply Theorem \ref{teo:MeanComp} to obtain
	
	\begin{equation*}
		\Delta^M\E_{r}^\omega\leq\,(\geq)\,\Delta^M E_{r}^M=-1\,\,\text{on}\,\, B^M_{r}.
			\end{equation*} 
			
			\noindent Therefore, since $\norm{\nabla^M r}=1$, and using the \emph{Divergence Theorem}, we have	
	\begin{equation}\label{eqeq1}
		\begin{split}
			\Vol\left(\brm_{r}^M\right) &\leq \int_{\brm_{r}^M}-\Delta^M\E_{r}^\omega\,d\widetilde{\sigma}
			 = -\int_{\brm_{r}^M}\Div\left(\nabla^M\E_{r}^\omega\right)\,d\widetilde{\sigma}\\&=-\int_{\srm_{r}^M}\ep{\nabla^M\E_{r}^\omega,\nabla^M r}\,d\sigma
			 =-\E_{r}^\omega{'}({r})\Vol\left(\srm_{r}^M\right).
		\end{split}
	\end{equation}
	
	Thus, we obtain, using Proposition \ref{prop:MeanTorEqualities},
	
	\begin{equation*}
		\Vol\left(\brm_{r}^M\right)\leq -\E_r^\omega{'}({r})\Vol\left(\srm_{r}^M\right)=q_\omega({r})\Vol\left(\srm_{r}^M\right)=\dfrac{\Vol\left(\brm_{r}^\omega\right)}{\Vol\left(\srm_{r}^\omega\right)}\Vol\left(\srm_{r}^M\right)
	\end{equation*}
	
	\noindent and therefore
	
	\begin{equation*}
		\dfrac{\Vol\left(\brm_{r}^\omega\right)}{\Vol\left(\srm_{r}^\omega\right)}\geq \dfrac{\Vol\left(\brm_{r}^M\right)}{\Vol\left(\srm_{r}^M\right)}.
	\end{equation*}
	
	We are going to discuss the equality assertion: we are still assuming that  $\hrm_{\srm_r^\omega}\leq\,\hrm_{\srm_r^M}$, for all $ 0<r \leq R$. If there exists $r_0 \in ]0,R]$ such that we have
	
	\begin{equation*}
		\dfrac{\Vol\left(\brm_{r_0}^\omega\right)}{\Vol\left(\srm_{r_0}^\omega\right)}= \dfrac{\Vol\left(\brm_{r_0}^M\right)}{\Vol\left(\srm_{r_0}^M\right)}
	\end{equation*}
	
	\noindent then all the inequalities in (\ref{eqeq1}) become equalities with the radius $r_0$. 
	
	In particular, 
	$$\Vol\left(\brm_{r_0}^M\right) = \int_{\brm_{r_0}^M}-\Delta^M\E_{r_0}^\omega\,d\widetilde{\sigma}$$
	\noindent and hence, as $1+\Delta^M\E_{r_0}^\omega \leq 0$ on $B^M_{r_0}$, we conclude that $1+\Delta^M\E_{r_0}^\omega = 0$ on $B^M_{r_0}$ and hence, as $\Delta^M\E_{r_0}^\omega=\Delta^ME^M_{r_0}$ on $B^M_{r_0}$ then, applying the maximum principle, $\E_{r_0}^\omega=E^M_{r_0}$ on $B^M_{r_0}$ and hence, by Proposition \ref{prop3.2}, $\hrm_{\srm_r^\omega}= \hrm_{\srm_r^M}\,\,\forall r \in ]0, r_0]$.
	
	When we assume that $\hrm_{\srm_r^\omega}\geq \hrm_{\srm_r^M}\,\,\forall r \in ]0, R]$, we argue as before, inverting all the inequalities to conclude the opposite isoperimetric inequality. The equality discussion is the same, mutatis mutandi.

	  To prove statement (\ref{ineqVolumes}), and as in Corollary 1.2 in \cite{Pa2}, let us define, given $0 <R < inj(o)$  the function $$G: [0,R] \rightarrow \erre$$
	\noindent  as 	
	\begin{equation*}
		 G(s):=
			\begin{cases}
				\ln\left(\dfrac{\Vol\left(\brm_s^M\right)}{\Vol\left(\brm_s^\omega\right)}\right),\quad\text{if } s>0,\\
				0,\qquad\qquad\qquad\quad\;\;\,\text{if } s=0.
			\end{cases}
		\end{equation*}
Then, if $\hrm_{\srm_s^\omega}\leq \hrm_{\srm_s^M} \,\,\forall s\in ]0, R]$, we have, applying  inequality (\ref{cor:MeanComp1}), that 
	
	\begin{equation*}
		%\begin{split}
			G'(s) =  \dfrac{\Vol\left(\srm_s^M\right)}{\Vol\left(\brm_s^M\right)}-\dfrac{\Vol\left(\srm_s^\omega\right)}{\Vol\left(\brm_s^\omega\right)}\geq\,0\,\,\forall s\in ]0,R].
		%\end{split}
	\end{equation*}
	
	Hence, $G$ is non-decreasing in $]0,R]$. The rest of the proof follows as in \cite{Pa2}, using in this case the asymptotic expansion around $s=0$ for the volume of a geodesic $s$-ball, (see Theorem 9.12 in \cite{Gray1})  to conclude with a straightforward computation, that $\lim_{s \to 0} G(s)=0=G(0)$, and hence, that $G(s)$ is continuous and $G(s) \geq G(0)\,\,\forall s \in [0, R]$, so, given $s=r \in ]0,R]$, we have
	
	\begin{equation*}
		\Vol\left(\brm_r^\omega\right)\leq \Vol\left(\brm_r^M\right)\,\,\forall  r \in ]0,R].
	\end{equation*}
	
	Moreover, isoperimetric inequality (\ref{cor:MeanComp1}), together above inequality implies that 
	
$$\Vol(S^\omega_r) \leq \Vol(S^M_r) \,\,\,\forall r\leq R.$$ 

We are going to discuss the equality assertion: let us assume that $\hrm_{\srm_r^\omega}\leq \hrm_{\srm_r^M}\,\,\forall r \in ]0, R]$ and that there exists $r_0 \in ]0,R]$ such that $\Vol\left(\brm_{r_0}^\omega\right)= \Vol\left(\brm_{r_0}^M\right)$. Then, $G(0)=G(r_0)=0$ and, as $G$ in non-decreasing, for all $r \in [0,r_0]$, we have
$$0=G(0) \leq G(r) \leq G(r_0)=0$$
\noindent so $G(r)=0 \,\,\forall r \in [0,r_0]$ and therefore, $G'(r)=0 \,\,\forall r \in [0,r_0]$ which implies that $\hrm_{\srm_r^\omega}= \hrm_{\srm_r^M}\,\,\forall r \in ]0, r_0]$.

	When we assume that $\hrm_{\srm_r^\omega}\geq \hrm_{\srm_r^M}\,\,\forall r \in ]0, R]$, we argue as before, inverting all the inequalities to conclude that $G$ is non-increasing in $]0,R]$ and hence 
	 \begin{equation*}
		\Vol\left(\brm_r^\omega\right)\, \geq\,\Vol\left(\brm_r^M\right)\,\,\forall  r \in ]0,R]
	\end{equation*}
	
	and $\Vol(S^\omega_r) \geq \Vol(S^M_r) \,\,\,\forall r\leq R$.
	
	The equality discussion is the same than above, mutatis mutandi.
	
\end{proof}

\begin{corollary}\label{cor:MeanComp2}

Let $(M^n,g)$ be a complete Riemannian manifold and let $(\MW,g_\omega)$ be a rotationally symmetric model space with center $o_\omega \in \MW$. Let $o \in M$ be a point in $M$ and let us suppose that  $inj(o) \leq inj(o_\omega)$. Let us consider a metric ball $B^M_R(o)$, with $R < inj(o) \leq inj(o_w)$. Let us suppose moreover that
	
	\begin{equation}\label{eq:meancurvatureconditions5}
		\hrm_{\srm_r^\omega}\leq\,(\geq)\, \hrm_{\srm_r^M}\quad\text{for all}\quad 0 < r \leq R.
	\end{equation}

	Then, if there exists $p \in B^M_R(o)$ such that equality $\E^\omega_R(p)=E^M_R(p)$ holds, we have, for all   $r \in ]0,R]$:
	
	\begin{enumerate}
		\item  The equalities $\E^\omega_r=E^M_r$ on $B^M_r(o)$.
		
		\item The isoperimetric equalities
		
		$$\dfrac{\Vol\left(\brm_r^\omega(o_\omega)\right)}{\Vol\left(\srm_r^\omega(o_\omega)\right)\,}=\,\dfrac{\Vol\left(\brm_r^M(o)\right)}{\Vol\left(\srm_r^M(o)\right)}.$$
		
		\item  The volume equalities  $\Vol\left(\brm_r^\omega\right)\, =\,\Vol\left(\brm_r^M\right)$
	and $\Vol(S^\omega_r) = \Vol(S^M_r) $.
	\end{enumerate}
\end{corollary}

\begin{proof}

 First of all, equality assertion in Theorem \ref{teo:MeanComp} states that, as we are assuming one of the inequalities in  (\ref{eq:meancurvatureconditions5}), then  if there exists $p \in B^M_R(o)$ such that equality $\E^\omega_R(p)=E^M_R(p)$ holds, we con conclude the equality $\E^\omega_R=E^M_R$ on $B^M_R(o)$ and applying Proposition \ref{prop3.2}, from this equality,  we have equality $\hrm_{\srm_r^\omega}= \hrm_{\srm_r^M} \,\,\forall r\in ]0, R]$. This last equality implies that, given any fixed $r \in ]0,R]$, we have the equalities $\hrm_{\srm_s^\omega}= \hrm_{\srm_s^M} \,\,\forall s\in ]0, r]$ and hence, by Proposition \ref{prop3.2} again, we obtain $\E^\omega_r=E^M_r$ on $B^M_r(o)$.

On the other hand, equality $\E^\omega_R=E^M_R\,\,\text{on}\,\, B^M_R(o)$ implies that $\Delta^M\E_R^\omega=-1=\Delta^M E^M_R$ on $B^M_R(o)$, which implies in its turn that
$$\Vol\left(\brm_{R}^M\right) = \int_{\brm_{R}^M}-\Delta^M\E_{R}^\omega\,d\sigma=-\E_{R}^\omega{'}({R})\Vol\left(\srm_{R}^M\right)$$
\noindent and hence, by Proposition \ref{prop:MeanTorEqualities}
$$\dfrac{\Vol\left(\brm_R^\omega(o_\omega)\right)}{\Vol\left(\srm_R^\omega(o_\omega)\right)\,}=\,\dfrac{\Vol\left(\brm_R^M(o)\right)}{\Vol\left(\srm_R^M(o)\right)}.$$

Moreover, fixing $r \in ]0,R]$, we know that, as $\E^\omega_R=E^M_R\,\,\text{on}\,\, B^M_R(o)$, then $\E^\omega_r=E^M_r\,\,\text{on}\,\, B^M_r(o)$ applying Proposition \ref{prop3.2}, and this equality implies equality
$$\dfrac{\Vol\left(\brm_r^\omega(o_\omega)\right)}{\Vol\left(\srm_r^\omega(o_\omega)\right)\,}=\,\dfrac{\Vol\left(\brm_r^M(o)\right)}{\Vol\left(\srm_r^M(o)\right)}$$
\noindent with the same argument than above.

\item Finally,  as equality $\E^\omega_R=E^M_R\,\,\text{on}\,\, B^M_R(o)$ implies that $\E^\omega_r=E^M_r\,\,\text{on}\,\, B^M_r(o)\,\,\forall r\in ]0,R]$, then, if we define \begin{equation*}
		 G(r):=
			\begin{cases}
				\ln\left(\dfrac{\Vol\left(\brm_r^M\right)}{\Vol\left(\brm_r^\omega\right)}\right),\quad\text{if } r\in ]0,R],\\
				0,\qquad\qquad\qquad\quad\;\;\,\text{if } r=0,
			\end{cases}
		\end{equation*}
		\noindent then $G'(r)=0 \,\forall r \in ]0,R]$, and hence, $G(r)=0 \,\forall r \in ]0,R]$, so $\Vol\left(\brm_r^\omega\right)\, =\,\Vol\left(\brm_r^M\right)\,\,\forall  r \in ]0,R]$ and differentiating with respect the parameter $r$, $\Vol(S^\omega_r) = \Vol(S^M_r) \,\,\,\forall r\leq R$.

\end{proof}

To finish this section, we present the following property satisfied by the symmetrization of the transplanted mean exit time function  $\E_R^\omega$. This result is an intrinsic version of Theorem 4.4 in \cite{HuMP1}, and it follows directly from this result, (see too Section 6 in \cite{HuMP1}). 

\begin{theorem}\label{cor:eqschwarz}
	Let $(M^n,g)$ be a complete Riemannian manifold and let $(\MW,g_\omega)$ be a rotationally symmetric model space with center $o_\omega \in \MW$. Let $o \in M$ be a point in $M$ and let us suppose that  $inj(o) \leq inj(o_\omega)$. Let us consider a metric ball $B^M_R(o)$, with $R < inj(o) \leq inj(o_w)$, and let us assume that there exists $B^\omega_{s(R)}(o_\omega)$, the Schwarz symmetrization of $B^M_R$ in $M^n_\omega$. Let $\E_R^\omega{^*}\,:\,\brm_{s(R)}^\omega\longrightarrow\R$ be the symmetrization of the transplanted mean exit time function  $\E_R^\omega\,:\,\brm_R^M\longrightarrow\R$. Then
	
	\begin{equation}\label{eq:corschwarz}
		\int_{\brm_R^M}\E_R^\omega\,d\sigma=\int_{\brm_{s(R)}^\omega}\E_R^\omega{^*}\,d\widetilde{\sigma}.
	\end{equation}
\end{theorem}

%%%%%%%%%%%%%%%%%%%%%%%%%%%%%%%%
%%%%%%%%%%%%%%%%%%%%%%%%%%%%%%%%
\section{Moment spectrum comparison}\label{sec:TorRidCom1}\
%%%%%%%%%%%%%%%%%%%%%%%%%%%%%%%%%%
\medskip

We are going to apply the Mean Exit comparisons obtained in Section 3 to obtain estimates of the moment spectrum, and the torsional rigidity  of a geodesic ball in a Riemannian manifold with bounds on the mean curvature of its extrinsic spheres.
%%%%%%%%%%%%%%%%%%%%%%%%%%%%%%%%%%%%%%%
%%%%%%%%%%%%%%%%%%%%%%%%%%%%%%%%%%%%%%%%
\subsection{Estimates for the Poisson hierarchy and the moment spectrum of a geodesic ball}\label{sec:MomentsComp}\
%%%%%%%%%%%%%%%%%%%%%%%%%%%%%%%%%%%%%%%
%%%%%%%%%%%%%%%%%%%%%%%%%%%%%%%%%%%%%%

\medskip

We shall start defining the so called  Poisson hierarchy of a domain in a Riemannian manifold, (see \cite{DLD}).
\begin{definition}
	Let $(M^n,g)$ be a complete Riemannian manifold and let $D\in M$ be a smooth precompact domain. We define the \emph{Poisson hierarchy for $D$} as the sequence $\{u_{k,D}\}_{k=1}^\infty$ of solutions of the following recurrence of boundary value problems
	
	%\begin{equation}\label{poisson}
		%\begin{split}
			%\Delta^Mu_1+1 & =  0,\qquad \text\,\,{on }\,\,D,\\
			%u_1\lvert_{_{\partial D}} & = 0,
		%\end{split}
	%\end{equation}
	
	%\noindent and, for $k\geq 2$,
	
	\begin{equation}\label{poissonk}
		\begin{split}
			\Delta^Mu_{k,D}+ku_{k-1,D} & =  0,\,\, \text{on }\,\,D,\\
			u_{k,D}\lvert_{_{\partial D}} & = 0,
		\end{split}
	\end{equation}
	\noindent with $u_{0,D}=1 \,\,\text{on}\,\, D$.
	
	Let us note that $u_{1,D}=E_D^M$, i.e. the mean exit time function from $D$.
	
\end{definition}

As we did in Definition \ref{transplanted}, we transplant the Poisson hierarchy for the geodesic balls in a model space to the geodesic balls in a Riemanian manifold in the following way:

\begin{definition}

Let $(M^n,g)$ be a complete Riemannian manifold and let $(\MW,g_\omega)$ be a rotationally symmetric model space with center $o_\omega \in \MW$. Let $o \in M$ be a point in $M$ and let us suppose that  $inj(o) \leq inj(o_\omega)$. Let us consider a metric ball $B^M_R(o)$, with $R < inj(o) \leq inj(o_w)$.

Let us consider  the Poisson hierarchy for $B^\omega_R(o_\omega)$, namely,  the sequence $\{u_{k,R}^\omega\}_{k=1}^\infty$  which, for $k\geq 1$, are the solutions of 
	
	\begin{equation*}
		\begin{split}
			\Delta^{\MW}u_{k,R}^\omega+ku_{k-1,R}^\omega & =0,\,\,\text{on }\brm_R^\omega,\\
			u_{k,R}^\omega\lvert_{_{\srm_R^\omega}}& = 0,
		\end{split}
	\end{equation*}
	\noindent with $u^\omega_{0,R}=1 \,\,\text{on}\,\, \brm_R^\omega$.
	
	It is known, for all $k\geq 1$ that $u_{k,R}^\omega(x)=u_{k,R}^\omega\left(r_{o_\omega}(x)\right)$, i.e. $u_{k,R}^\omega$ is radial, and that $u_{k,R}^\omega{'}\leq 0$ (see Proposition 3.1 of \cite{HuMP2}). 
	
	Thus, for all $k\geq 1$, we can transplant these functions to $\brm_R^M(o) \subseteq M$ by defining 
	$$\bu_{k,R}^{\omega}: \brm_R^M(o) \rightarrow \erre$$ 
	\noindent as $\bu_{k,R}^{\omega}(x):=u_{k,R}^\omega\left(r_o(x)\right)\,\,\forall x \in \brm_R^M(o)$, where $r_o$ is the distance function to the center of $\brm_R^M(o)$. 
	
	The sequence $\{\bu_{k,R}^\omega\}_{k=1}^\infty$ is the {\em transplanted Poisson hierarchy for $B^M_R(o)$}.

\end{definition}

Associated to the Poisson hierarchy of a domain $D \subseteq M$ it is defined the exit time moment spectrum of this domain in the following way:

\begin{definition}
	Let $D\subseteq M$ a smooth precompact domain.  We define the \emph{moment spectrum of $D$} as the sequence of integrals $\{\CA_k(D)\}_{k=1}^{\infty}$ given by:
	
	\begin{equation*}
		\CA_k(D):=\int_D u_{k,D}\,d\sigma\, ,
	\end{equation*}
	
	\noindent where $\{u_{k,D}\}_{k=1}^\infty$ is the Poisson hierarchy for $D$.
	
	Let us note that  $\CA_1(D)$ is the torsional rigidity of $D$.
\end{definition}

We have the following comparison for the Poisson hierarchy of a geodesic ball in a Riemannian manifold:

\begin{theorem}\label{teo:TowMomComp}
	Let $(M^n,g)$ be a complete Riemannian manifold and let $(\MW,g_\omega)$ be a rotationally symmetric model space with center $o_\omega \in \MW$. Let $o \in M$ be a point in $M$ and let us suppose that  $inj(o) \leq inj(o_\omega)$. Let us consider a metric ball $B^M_R(o)$, with $R < inj(o) \leq inj(o_w)$. Let us suppose moreover that the mean curvatures of the geodesic spheres in $M$ and $\MW$ satisfies
	
	\begin{equation}\label{eq:meancurvatureconditions6}
		\hrm_{\srm_r^\omega}\leq\,(\geq)\, \hrm_{\srm_r^M}\quad\text{for all}\quad 0 < r \leq R.
	\end{equation}
	
	%\noindent where $\hrm_{\srm_r^M}$, denotes the mean curvature of the distance $r$- sphere $\srm_r^M(o) \subseteq M$ and $\hrm_{\srm_r^\omega}$ is the corresponding mean curvature of the distance $r$-sphere $\srm_r^\omega(o_\omega) \subseteq \MW$. 
	
	Then the Poisson hierarchy for $\brm_R^M(o)\subseteq M$,  $\{u_{k,R}\}_{k=1}^\infty$, and its transplanted Poisson hierarchy for $B^M_R(o)$, $\{\bu_{k,R}\}_{k=1}^\infty$ (and, fixed any $r \in ]0, R]$, the corresponding Poisson hierarchies for $B^M_r(o)$), satisfies
	
	\begin{enumerate}
		\item \label{teo:TowMomComp1} $\bu_{1,R}^{\omega}\geq\,(\leq)\,u_{1,R}$ on $\brm_R^M$.

		\item \label{teo:TowMomComp2}  For all $k\geq 2$, $\bu_{k,R}^{\omega}\geq\,(\leq)\,u_{k,R}$ on $\brm_R^M$.		
		% where $\{\bu_{k,R}^\omega\}_{k=1}^\infty$, (resp. $\{\bu_{k,r}^\omega\}_{k=1}^\infty$) is the {\em transplanted Poisson hierarchy for $B^M_R(o)$}, (resp. the {\em transplanted Poisson hierarchy for $B^M_r(o)$}).
	
	\item If there exists $p \in B^M_R$ and $k_0 \geq 1$ such that $\bu_{k_0, R}^{\omega}(p)=u_{k_0,R}(p)$, then 
	$$\hrm_{\srm_r^\omega}=\hrm_{\srm_r^M} \, \forall r \in ]0,R]$$
	\noindent and
	$$\bu_{k,R}^{\omega}=u_{k,R}\,\text{ in} \,\,\,B^M_R\, \,\forall \, k \geq 1.$$ 
	
	\item If there exists $p \in B^M_R$ and $k_0 \geq 1$ such that $\bu_{k_0, R}^{\omega}(p)=u_{k_0,R}(p)$, then
	
	$$\bu_{k,r}^{\omega}=u_{k,r}\,\text{ in} \,\,\,B^M_r\, \,\forall \, k \geq 1\,\,\,\text{and}\,\,\,\forall r \in [0,R].$$ 
	
	\noindent and hence,
	$$\CA_k(B^\omega_r)=\CA(B^M_r)\,\,\forall r \in ]0,R]\,\,\text{and}\,\,\, k \geq 1\, .$$

	\end{enumerate}
\end{theorem}

\begin{proof}

	Statement \eqref{teo:TowMomComp1} is proved in Theorem \ref{teo:MeanComp}. 
	
	The proof of  statement \eqref{teo:TowMomComp2} follows using induction on $k$, as it is done in \cite{HuMP2}. Indeed, assuming that $\hrm_{\srm_r^\omega}\leq \hrm_{\srm_r^M}\,\,\forall r \in ]0,R]$ and as $\bu_{k,R}^{\omega}{'}(r)\leq 0\,\,\forall r \in ]0,R]$, we have that
	
	\begin{equation*}
		\bu_{k,R}^{\omega}{'}(r)\hrm_{\srm_r^\omega}\geq\bu_{k,R}^{\omega}{'}(r)\hrm_{\srm_r^M}\,\,\forall r \in ]0,R]
	\end{equation*}
	and then, by equations  (\ref{eq:WarpMean}) and (\ref{laplacemodel}) and Proposition \ref{lapmean}, we have, for all $k\geq 2$:
	
	\begin{equation}\label{eq:TowMomLap}
		\begin{aligned}
			\Delta^M\bu_{k,R}^{\omega} & =\bu_{k,R}^{\omega}{''}(r)+(n-1)\hrm_{\srm_r^M}\,\bu_{k,R}^{\omega}{'}(r)\leq\bu_{k,R}^{\omega}{''}(r)+(n-1)\hrm_{\srm_r^\omega}\bu_{k,R}^{\omega}{'}(r)\\& =\Delta^{\MW}u_{k,R}^\omega=-ku_{k-1}^\omega(r)=-k\bu_{k-1}^{\omega}(r)\,\,\forall r \in ]0,R].
		\end{aligned}
	\end{equation}

	Now remember that  $\bu_1^{\omega}\geq u_1$ on $\brm_R^M$ and let us suppose that 
	$$\bu_{k,R}^{\omega}\geq u_{k,R} \,\,\text{on} \,\,\brm_R^M.$$
	
	Then, by induction with $k+1$ and using equation (\ref{eq:TowMomLap}), we have that
	
	\begin{equation}\label{eq:TowMomCompLapDes}
		\Delta^M\bu_{k+1,R}^{\omega}\leq-(k+1)\bu_{k,R}^{\omega}\leq-(k+1)u_{k,R}=\Delta^M u_{k+1,R}\,\,\text{on}\,\,\, \brm_R^M.
	\end{equation}
	
	Thus, $\Delta^M\left(u_{k+1,R}-\bu_{k+1,R}^{\omega}\right) \geq 0$ on $\brm_R^M$ and, applying the Maximum Principle, we obtain that
	$$\bu_{k+1,R}^{\omega}\geq u_{k+1,R}\, .$$
	
	When we assume that $\hrm_{\srm_r^\omega}\geq\,\hrm_{\srm_r^M}\,\,\forall r \in ]0,R]$, the argument is exactly the same, inverting all the inequalities. All this proves \eqref{teo:TowMomComp2}.
	
	We are going to prove assertion (3). Let us suppose that, as hypothesis, $\hrm_{\srm_r^\omega}\leq \hrm_{\srm_r^M}\,\,\forall r \in ]0,R]$, and that there exists $p \in B^M_R$ and $k_0 \geq 1$ such that $$\bu_{k_0,R}^{\omega}(p)=u_{k_0,R}(p).$$
	
	 We know that, for all $k \geq 1$, $\bu_{k,R}^{\omega}\geq u_{k,R} \,\,\text{on} \,\,\brm_R^M$. Then, as, on $\brm^M_R(o)$,  $\Delta^M \bu_{k,R}^{\omega} \leq -k\bu_{k-1}^{\omega}\leq -k\ u_{k-1}=\Delta^M u_{k,R}$ for all $k \geq 1$, we have, in particular,
	$$\Delta^M\left(u_{k_0,R}-\bu_{k_0,R}^{\omega}\right) \geq 0$$
	\noindent  on $\brm_R^M(o)$.
	
	Moreover, as $\bu_{k_0,R}^{\omega}\geq u_{k_0,R}$ on $\brm_R^M(o)$, then $u_{k_0,R}-\bu_{k_0,R}^{\omega} \leq 0$ on $\brm_R^M(o)$ and there exists $p \in B^M_R$  such that $(u_{k_0,R}-\bu_{k_0,R}^{\omega})(p)=0$. Then, applying the strong maximum principle, $\bu_{k_0,R}^{\omega}= u_{k_0,R}$ on $\brm_R^M$, because $u_{k_0,R}-\bu_{k_0,R}^{\omega}$ is constant on $\brm^M_R(o)$, continuous in $\overline{\brm^M_R(o)}$ and $u_{k_0,R}-\bu_{k_0,R}^{\omega}=0$ on $\srm^M_R(o)$.

	On the other hand,  as $\bu_{k_0-1}^{\omega}\geq u_{k_0-1}$ on $\brm_R^M$, we have, on $B^M_R(o)$:
	
	\begin{equation}
		\begin{split}
			\Delta^M\bu_{k_0,R}^{\omega}  &=\Delta^M u_{k_0,R}=-k_0 u_{k_0-1,R}\geq\\& -k_0 \bu_{k_0-1,R}^{\omega}=-k_0u_{k_0-1,R}^\omega=\Delta^{\MW}u_{k_0,R}^\omega
		\end{split}
	\end{equation}
	
	\noindent so, for all $r \in ]0,R]$:
	
	\begin{equation}
	\bu_{k_0,R}^{\omega}{''}(r)+(n-1)\hrm_{\srm_r^M}\,\bu_{k_0,R}^{\omega}{'}(r) \geq u_{k_0,R}^{\omega}{''}(r)+(n-1)\hrm_{\srm_r^\omega}u_{k_0,R}^{\omega}{'}(r).
	\end{equation}
	
	As $\bu_{k_0,R}^{\omega}{''}(r)=u_{k_0,R}^{\omega}{''}(r)$ and $\bu_{k_0,R}^{\omega}{'}(r)=u_{k_0,R}^{\omega}{'}(r)$ for all $r \in ]0,R]$, we conclude that 
	
	$$ \hrm_{\srm_r^M}\bu_{k_0,R}^{\omega}{'}(r) \geq \hrm_{\srm_r^\omega} u_{k_0,R}^{\omega}{'}(r) \,\,\forall r \in ]0,R]$$
	\noindent and hence, as $u_{k_0,R}^{\omega}{'}(r) < 0\,\,\forall r \in ]0,R]$, then
	$$\hrm_{\srm_r^\omega}\geq \hrm_{\srm_r^M}\,\,\forall r \in ]0,R].$$
	
	As, by hypothesis, $\hrm_{\srm_r^\omega}\leq \hrm_{\srm_r^M}\,\,\forall r \in ]0,R]$, we have finally that 
	$$\hrm_{\srm_r^\omega}= \hrm_{\srm_r^M}\,\,\forall r \in ]0,R].$$
	
	Now, to prove that $\bu_{k,R}^{\omega}=u_{k,R}$ on $B^M_R(o)$, we argue as follows:  as we know that $\hrm_{\srm_r^\omega}= \hrm_{\srm_r^M}\,\,\forall r \in ]0,R]$, let us apply Proposition \ref{prop3.2}, to have that $\bu_{1,R}^{\omega}=u_{1,R}$ on $B^M_R(o)$, and we procceed by induction: let us suppose that $\bu_{k,R}^{\omega}=u_{k,R}$ on $B^M_R(o)$. Let us see that $\bu_{k+1,R}^{\omega}=u_{k+1,R}$ on $B^M_R(o)$. For that, we compute
	
	\begin{equation}
	\begin{split}
			\Delta^M\bu_{k+1,R}^{\omega} & =\bu_{k+1,R}^{\omega}{''}(r)+\hrm_{\srm_r^M}\,\bu_{k+1,R}^{\omega}{'}(r) =\bu_{k+1,R}^{\omega}{''}(r)+\hrm_{\srm_r^\omega}\bu_{k+1,R}^{\omega}{'}(r)\\& =\Delta^{\MW}u_{k+1,R}^\omega=-(k+1)u_{k,R}^\omega=-(k+1)\bu_{k,R}^{\omega}\\&=-(k+1)u_{k,R}=\Delta^M u_{k+1,R}\,\,\text{on}\,\,B^M_R(o).
		\end{split}
	\end{equation}
	
	Hence $\Delta^M (\bu_{k+1,R}^{\omega}-u_{k+1,R})=0$ on $B^M_R(o)$ and as $\bu_{k+1,R}^{\omega}-u_{k+1,R}=0$ on $S^M_R(o)$, then, appliyng Maximum Principle again, we conclude that $\bu_{k+1,R}^{\omega}=u_{k+1,R}=0$ on $B^M_R(o)$.
	
	Finally, to prove assertion (4), let us assume that that $\hrm_{\srm_r^\omega}\leq \hrm_{\srm_r^M}\,\,\forall r \in ]0,R]$, and that there exists $p \in B^M_R$ and $k_0 \geq 1$ such that $$\bu_{k_0,R}^{\omega}(p)=u_{k_0,R}(p).$$
	\noindent As before, we conclude that 
	$$\hrm_{\srm_r^\omega}= \hrm_{\srm_r^M}\,\,\forall r \in ]0,R],$$
	\noindent and hence, fixing $r \in ]0,R]$, that 
	$$\hrm_{\srm_s^\omega}= \hrm_{\srm_s^M}\,\,\forall s \in ]0,r].$$
	\noindent Now, to prove that $\bu_{k,r}^{\omega}=u_{k,r}$ on $B^M_r(o)$, we argue as in the proof of (3): as we know that $\hrm_{\srm_s^\omega}= \hrm_{\srm_s^M}\,\,\forall s \in ]0,r]$, let us apply Proposition \ref{prop3.2}, to have that $\bu_{1,r}^{\omega}=u_{1,r}$ on $B^M_r(o)$, and we procceed by induction, as in the proof of assertion (3).

\end{proof}

As a consequence of the Theorem \ref{teo:TowMomComp} we have the following result, where it is proved that, under our hypotheses, any of the averaged moments of the geodesic balls determines its first Dirichlet eigenvalue:

\begin{corollary}\label{isoptors}
	Let $(M^n,g)$ be a complete Riemannian manifold and let $(\MW,g_\omega)$ be a rotationally symmetric model space with center $o_\omega \in \MW$. Let $o \in M$ be a point in $M$ and let us suppose that  $inj(o) \leq inj(o_\omega)$. Let us consider a metric ball $B^M_R(o)$, with $R < inj(o) \leq inj(o_w)$. Let us suppose moreover that the mean curvatures of the geodesic spheres in $M$ and $\MW$ satisfies
	
	\begin{equation}\label{eq:meancurvatureconditions7}
		\hrm_{\srm_r^\omega}\leq\,(\geq)\, \hrm_{\srm_r^M}\quad\text{for all}\quad 0 < r \leq R.
	\end{equation}
	
Then, for all $k\geq 1$,
	
	\begin{equation}\label{RigIsop}
		\dfrac{\CA_k\left(\brm_R^\omega\right)}{\Vol\left(\srm_R^\omega\right)}\geq\,(\leq)\,\dfrac{\CA_k\left(\brm_R^M\right)}{\Vol\left(\srm_R^M\right)}.
	\end{equation}
	
	Equality in any of  inequalities (\ref{RigIsop}) for some $k \geq 1$ implies  that 
	 $$\hrm_{\srm_R^\omega(o_\omega)}=\,\hrm_{\srm_R^M(o)}\,\,\,\text{for all }0<r\leq R$$
	 \noindent and hence, we have the equalities
	
	\begin{enumerate}
	
	\item Equality $\bu_{k,R}^{\omega}=u_{k,R}$ on $B^M_R(o)$ for all $k\geq 1$, and hence, equality $\bu_{k,r}^{\omega}=u_{k,r}$ on $B^M_r(o)$ for all $k\geq 1$ and for all $0 < r \leq R$.
	\item Equalities $\Vol\left(\brm_r^\omega\right)=\Vol\left(\brm_r^M\right)$ and $\Vol\left(\srm_r^\omega\right)=\Vol\left(\srm_r^M\right)\,\,\,\text{for all }\, 0<r \leq R$.
	\item Equalities $\CA_k\left(\brm_r^\omega\right)=\CA_k\left(\brm_r^M\right)$,  for all $k\geq 1$ and for all $0<r \leq R$.
	\item Equalities 
	$\lambda_1 (B^\omega_{r})=\lambda_1(B^M_{r})$ for all $0<r \leq R$.
	
	\noindent Namely, {\em one} value of $\dfrac{\CA_k\left(\brm_R^M(o)\right)}{\Vol\left(\srm_R^M(o)\right)}$ for some $k \geq 1$ determines the Poisson hierarchy, the volume, the $L^1$-moment spectrum and the first Dirichlet eigenvalue of the ball $B^M_r(o)$ for all $0<r \leq R$.
	\end{enumerate}
	
\end{corollary}

\begin{proof}

In the model spaces we have that $\Delta^{M_{\omega}} u_{k+1,R}^{\omega}=-(k+1) u_{k}^{\omega}$ on the geodesic ball $\brm^{M_{\omega}}_R(o_{\omega})$, so, applying Divergence theorem in this setting ,  we obtain 
	
	\begin{equation*}
		\CA_k\left(\brm_R^\omega\right)=\int_{\brm_R^\omega} u_{k,R}^\omega\,d\widetilde{\sigma}=-\dfrac{1}{k+1}\int_{\brm_R^\omega}\Delta^{\MW}u_{k+1,R}^\omega\,d\widetilde{\sigma}=-\dfrac{1}{k+1}\,u_{k+1,R}^\omega{'}(R)\Vol\left(\srm_R^\omega\right).
	\end{equation*}
	
	\noindent Therefore, for all $k \geq 1$,
	
	\begin{equation}\label{eq:TowMomDerviadaMW}
		-\dfrac{1}{k+1}u_{k+1,R}^\omega{'}(R)=\dfrac{\CA_k\left(\brm_R^\omega\right)}{\Vol\left(\srm_R^\omega\right)}.
	\end{equation}
	
	Assuming now as hypothesis one of the inequalities in (\ref{eq:meancurvatureconditions7}), we obtain correspondingly the inequalities 
$$\Delta^M\bu_{k+1,R}^{\omega}\leq\,(\geq)\,\Delta^M u_{k+1,R}\,\,\text{on}\,\,\, \brm_R^M.$$
	
	Then, using the Divergence theorem and that $\bu_{k+1,R}^{\omega}$ is radial in $\brm_R^M$, we have
	
	\begin{equation}\label{ineqisopTors}
		\begin{split}
			\CA_k\left(\brm_R^M\right) & =\int_{\brm_R^M}u_{k,R}\,d\sigma=-\dfrac{1}{k+1}\int_{\brm_R^M}\Delta^M u_{k+1,R}d\sigma\\&
			 \leq\,(\geq)\,-\dfrac{1}{k+1}\int_{\brm_R^M}\Delta^M\bu_{k+1,R}^{\omega}\,d\sigma\\&=-\dfrac{1}{k+1}\int_{\srm_R^M}\ep{\nabla^M\bu_{k+1,R}^{\omega},\nabla^Mr}\,d\sigma_r \\&= -\dfrac{1}{k+1}\,\bu_{k+1,R}^{\omega}{'}(R)\Vol\left(\srm_R^M\right).
		\end{split}
	\end{equation}
	
	Then, using equation \eqref{eq:TowMomDerviadaMW}, and that $\bu_{k+1,R}^\omega{'}(R)=u_{k+1,R}^\omega{'}(R)$, we finally obtain that
	
	\begin{equation*}
		\CA_k\left(\brm_R^M\right)\leq\,(\geq)\,\dfrac{\CA_k\left(\brm_R^\omega\right)}{\Vol\left(\srm_R^\omega\right)}\,\Vol\left(\srm_R^M\right).
	\end{equation*}
	
	We are going to discuss the equality case: assuming that $\hrm_{\srm_r^\omega}\leq\,\hrm_{\srm_r^M},\, \text{for all }0<r\leq R$, equality 
	\begin{equation}
		\dfrac{\CA_{k_0}\left(\brm_R^\omega\right)}{\Vol\left(\srm_R^\omega\right)}=\,\dfrac{\CA_{k_0}\left(\brm_R^M\right)}{\Vol\left(\srm_R^M\right)}
	\end{equation}
	\noindent for some $k_0 \geq 1$ implies that all inequalities in (\ref{ineqisopTors}) are equalities for this fixed $k_0$, so $\bu_{k_0+1,R}^{\omega}=u_{k_0+1,R}$ on $B^M_R(o)$.
	Applying assertion (3) in Theorem \ref{teo:TowMomComp}, we have that $\hrm_{\srm_r^\omega}=\,\hrm_{\srm_r^M}\,\,\,\text{for all }0<r \leq R$ and that $\bu_{k,R}^{\omega}=u_{k,R}$ on $B^M_R(o)$ for all $k\geq 1$. In particular, $\bu_{1,R}^{\omega}=u_{1,R}$ on $B^M_R(o)$, so, by Corollary \ref{cor:MeanComp2}, 
	$\Vol\left(\srm_r^\omega\right)=\Vol\left(\srm_r^M\right)$ and $\Vol\left(\brm_r^\omega\right)=\Vol\left(\brm_r^M\right)$ for all $r \in ]0, R]$ and, hence,  for all $k \geq 1$,
	
	\begin{equation}\label{ineqisopTors2}
		\begin{split}
			\CA_k\left(\brm_R^M\right) & =\int_{\brm_R^M}u_{k,R}\,d\sigma=-\dfrac{1}{k+1}\int_{\brm_R^M}\Delta^M u_{k+1,R}d\sigma
			 \\&=\,-\dfrac{1}{k+1}\int_{\brm_R^M}\Delta^M\bu_{k+1,R}^{\omega}\,d\sigma= -\dfrac{1}{k+1}\,\bu_{k+1,R}^{\omega}{'}(R)\Vol\left(\srm_R^M\right)\\ &=\dfrac{\CA_k\left(\brm_R^\omega\right)}{\Vol\left(\srm_R^\omega\right)}\,\Vol\left(\srm_R^M\right)= \CA_k\left(\brm_R^\omega\right).
		\end{split}
	\end{equation}
	
	Moreover, applying assertion (4) in Theorem \ref{teo:TowMomComp}, from equality $\bu_{k_0+1,R}^{\omega}=u_{k_0+1,R}$ on $B^M_R(o)$ we can deduce that $\bu_{k,r}^{\omega}=u_{k,r}$ on $B^M_r(o)$ for all $k\geq 1$ and for all $r \in ]0,R]$, so given $r \in ]0,R]$, and for all $k \geq 1$,
	
	\begin{equation}\label{ineqisopTors3}
		\begin{split}
			\CA_k\left(\brm_r^M\right) & =\int_{\brm_r^M}u_{k,r}\,d\sigma=-\dfrac{1}{k+1}\int_{\brm_r^M}\Delta^M u_{k+1,r}d\sigma
			 \\&=\,-\dfrac{1}{k+1}\int_{\brm_r^M}\Delta^M\bu_{k+1,r}^{\omega}\,d\sigma= -\dfrac{1}{k+1}\,\bu_{k+1,r}^{\omega}{'}(r)\Vol\left(\srm_r^M\right)\\ &=\dfrac{\CA_k\left(\brm_r^\omega\right)}{\Vol\left(\srm_r^\omega\right)}\,\Vol\left(\srm_r^M\right)= \CA_k\left(\brm_r^\omega\right).
		\end{split}
	\end{equation}
	
	%Hence, $\CA_k\left(\brm_r^\omega\right)=\CA_k\left(\brm_r^M\right)$,  for all $k\geq 1$.

Finally, to prove the last assertion of the Theorem, we know that, assuming that $\hrm_{\srm_r^\omega}\leq \hrm_{\srm_r^M}\,\,\text{for all}\,\,0 < r \leq R$, the equality 
\begin{equation*}
		\dfrac{\CA_{k_0}\left(\brm_R^\omega\right)}{\Vol\left(\srm_R^\omega\right)}=\dfrac{\CA_{k_0}\left(\brm_R^M\right)}{\Vol\left(\srm_R^M\right)}
	\end{equation*}
	\noindent  implies equalities  $\CA_k\left(\brm_r^\omega\right)=\CA_k\left(\brm_r^M\right)$,  for all $k\geq 1$, and for all $r \in ]0,R]$. 
Then, given $B^M_r \subseteq M$ in a Riemannian manifold $(M,g)$, (see \cite{HuMP3} and \cite{BGJ}):
\begin{equation}\label{gim2}
\begin{aligned}
\lambda_1(B^M_r)&=\lim_{k \to \infty}\dfrac{k\CA_{k-1}\left(\brm_r^M\right)}{\CA_{k}\left(\brm_r^M\right)}\\&=\lim_{k \to \infty}\dfrac{k\CA_{k-1}\left(\brm_r^\omega\right)}{\CA_{k}\left(\brm_r^\omega\right)}=\lambda_1(B^\omega_r).
\end{aligned}
\end{equation}

\end{proof}

%%%%%%%%%%%%%%%%%%%%%%%%%
%%%%%%%%%%%%%%%%%%%%%%%%%%%%
\subsection{An estimate for the torsional rigidity of a geodesic R-ball }\label{subsec:TRcomp}\
%%%%%%%%%%%%%%%%%%%%%%%%%%%
%%%%%%%%%%%%%%%%%%%%%%%%%%%%

\medskip

We are going to bound the torsional rigidity of a metric ball $\brm_R^M$ in a Riemannian manifold $(M,g)$ in Theorem \ref{teo:2.1}, assuming  that the mean curvature of the geodesic spheres in this Riemnnian manifold is bounded from above or from below by the corresponding mean curvature of the geodesic spheres in a symmetric model space $(\MW, g_{\omega})$ which is balanced from above.  

This result can be considered as a continuation of the  intrinsic comparison done in Section 6 of the paper \cite{HuMP1}. In that paper it were obtained upper and lower bounds for the torsional rigidity of a metric ball $\brm_R^M(o)$ in a Riemannian manifold $(M,g)$ with a pole $o \in M$ under more restrictive conditions, namely, assuming that  the radial sectional curvatures were bounded above or below by the corresponding radial sectional curvatures of a suitable model space.

To do that, let us consider a symmetric model space rearrangement of the metric ball $\brm_R^M$ as it has been described in Definition \ref{def:Symm} and Definition \ref{def:simfunc}, namely, a symmetrization of $\brm_R^M$ which is a geodesic $s(R)$-ball in the model space $\MW$ such that $\Vol\left(\brm_R^M(o)\right)=\Vol\left(\brm_{s(R)}^\omega(o_\omega)\right)$, together the symmetrization $\E_R^\omega{^*}\,:\,\brm_{s(R)}^\omega\longrightarrow\R$ of the transplanted mean exit time function $\E_R^\omega\,:\,\brm_R^M\longrightarrow\R$. It is evident that Proposition \ref{prop:ineqpsiE}, Theorem \ref{teo:2.1} and Corollary \ref{cor:2.1} make sense for those geodesic balls $B^M_R(o)$ which posses a Schwarz symmerization $\brm_{s(R)}^\omega(o_\omega)$.

Then, we have the following  comparison. Its proof follows closely the lines of the proof of Propositions 5.2 and 5.4 in \cite{HuMP1}, we have included it because the changes due to its intrinsic character,  the different assumptions on the curvatures we have assumed here and the new analysis of the equality we present in this case.

\begin{proposition}\label{prop:ineqpsiE}
	
	Let $(M^n,g)$ be a complete Riemannian manifold and let $(\MW,g_\omega)$ be a rotationally symmetric model space with center $o_\omega \in \MW$, balanced from above. Let $o \in M$ be a point in $M$ and let us suppose that  $inj(o) \leq inj(o_\omega)$. Let us consider a metric ball $B^M_R(o)$, with $R < inj(o) \leq inj(o_w)$. Let us suppose moreover that the mean curvatures of the geodesic spheres in $M$ and $M_{\omega}$ satisfies
	
	\begin{equation}\label{eq:meancurvatureconditions8}
		\hrm_{\srm_r^\omega}\leq\,(\geq)\, \hrm_{\srm_r^M}\quad\text{for all}\quad 0 < r \leq R.
	\end{equation}
	
	%\noindent where $\hrm_{\srm_r^M}$, denotes the mean curvature of the geodesic $r$- sphere $\srm_r^M(o) \subseteq M$ and $\hrm_{\srm_r^\omega}$ is the corresponding mean curvature of the distance $r$-sphere $\srm_r^\omega(o_\omega) \subseteq \MW$. 	
	Then

		\begin{equation}\label{eq:ineqpsiE}
			\E_R^\omega{^*}{'}(\widetilde{r})\geq\,(\leq)\,E_{s(R)}^\omega{'}(\widetilde{r})\quad\text{for all }\,\,\widetilde{r}\in(0,s(R))
		\end{equation}
		
		\noindent and hence, 
		\begin{equation}\label{eq:ineqpsiE2}
			\E_R^\omega{^*}(\widetilde{r})\leq\,(\geq)\,E_{s(R)}^\omega(\widetilde{r})\quad\text{for all }\,\,\widetilde{r}\in[0,s(R)].
		\end{equation}

\noindent Equality in any of the  inequalities (\ref{eq:ineqpsiE2}) implies the equality among the radius $s(R)=R$ and the equality
$$\hrm_{\srm_r^\omega}=\,\hrm_{\srm_r^M},\;\,\text{for all }0<r\leq R$$
\noindent and hence, we have the equalities among the  volumes $\Vol\left(\brm_{r}^\omega \right)=\Vol\left(\brm_{r}^M\right)\,\,\forall r \in ]0,R]$ and $\Vol\left(\srm_r^\omega\right)=\Vol\left(\srm_r^M\right)\,\,\,\text{for all }\, 0<r \leq R$.

	\end{proposition}
	
	%\begin{remark}
	%As we mentioned in the Introduction, the hypothesis on the mean curvatures of geodesic spheres (\ref{eq:meancurvatureconditions8}) can be replaced in above statement, (and in Theorem \ref{teo:2.1}), by the more general assumption that the following isoperimetric inequalities are satisfied by the geodesic $r$- balls $B^M_r(o)$ and its boundaries $S^M_r(o)$, with $r \leq R$, in the Riemannian manifold $M$, 
%\begin{equation} \label{hypisop1}
	%		\dfrac{\Vol\left(\brm_r^\omega(o_\omega)\right)}{\Vol\left(\srm_r^\omega(o_\omega)\right)}\geq\,(\leq)\,\dfrac{\Vol\left(\brm_r^M(o)\right)}{\Vol\left(\srm_r^M(o)\right)}\quad\text{for all}\quad 0<r \leq R
	%	\end{equation}
		
		%In fact, in the proof below, we can assume inequalities (\ref{hypisop1}) to obtain equation (\ref{eq.4.14}), and, after inequality (\ref{eq.4.16}),  to obtain inequalities (\ref{eq.4.17}), (\ref{eq.4.18}), and (\ref{eq.4.1}). With this, the proof of Proposition \ref{prop:ineqpsiE}, and subsequently, Theorem \ref{teo:2.1}, is done.
	%\end{remark}

	\begin{proof}

	We are going to analyze first the symmetrization $\E_R^\omega{^*}$. The transplanted function 
	$$\E_R^\omega:B_R^M(o)\longrightarrow\R$$
	 \noindent satisfies that $\E_R^\omega \in C^\infty ( B_R^M(o) \sim \{o\}) \cap C^0(\overline{B}_R^M(o))$, and, moreover, that $\E_R^\omega\vert_{S^M_R(o)}=0$.

	Let us consider the radial function $\psi=E_R^\omega$ defined on the interval $[0,R]$ in equation (\ref{eq:EwR}) of Proposition \ref{prop:MeanTorEqualities}. Let us denote by $T=\max_{[0,R]}\psi$. Thus, as $\psi$ is monotone, (strictly decreasing, with $\psi(0)=T$ and $\psi(R)=0$), we have that $\frac {d}{dr}\psi <0$ on $]0,R]$ and that $\psi:[0,R]\longrightarrow [0,T]$ is bijective. 
	
	Now,  let us define the function $a:[0,T]\longrightarrow[0,R]$ as $a(t):=\psi^{-1}(t)$, satisfying $a(0)=\psi^{-1}(0)=R$ and $a(T)=\psi^{-1}(T)=0$. We know that
	$$a'(t)=\frac{1}{\psi'(a(t))} <0\,\,\forall t \in (0,T)$$
	\noindent so $a(t)$ is strictly decreasing in $(0,T)$.

	Let us denote, for all $x \in B^M_R(o)$, 
	$$\varphi(x)=\E^\omega_R(x):=E^\omega_R(r_o(x))=\psi(r_o(x)).$$
	
	 We have that $\varphi(B^M_R(o))=\psi([0,R])=[0,T]$, so the function $\varphi: B^M_R(o) \rightarrow [0,T]$ satisfies  $\Vert \nabla^M \varphi\Vert=\vert \frac {d}{dr}\psi\vert \Vert \nabla^M r_o\Vert \neq 0$ for all $x \in B^M_R(o) -\{o\}$. Therefore, the set of regular values of $\varphi$ is $R_{\varphi}=(0,T)$.
	
	%On the other hand, the only critical point of $\psi$ is is $o \in B^M_R(o)$, so $R_\psi=]0,T]$, and hence, by Sard's theorem,  the set $\Gamma(t)=\left\{x\in\brm_R^M\,|\,\psi\left(r_o(x)\right)=t\right\}$ is a smooth embedded hypersurface and $\Vert \nabla^M\psi\Vert$ does not vanish in $\Gamma(t)$ for all $t \in ]0,T]$

	On the other hand, and given $t\in[0,T]$, let us consider the sets
	
	\begin{equation*}
	\begin{aligned}
		D(t)&=\left\{x\in\brm_R^M\,|\,\varphi(x)\geq t\right\}=\left\{x\in\brm_R^M\,|\,\E^\omega_R(x)\geq t\right\}\\&=\left\{x\in\brm_R^M\,|\,r_o(x)\leq \psi^{-1}(t)\right\}=\brm_{a(t)}^M
		\end{aligned}
	\end{equation*}
	
	and
	
	\begin{equation*}
	%\begin{aligned}
		\Gamma(t)=\left\{x\in\brm_R^M\,|\,\varphi(x)= t\right\}=\left\{x\in\brm_R^M\,|\,\psi\left(r_o(x)\right)=t\right\}=\srm_{a(t)}^M.
		%\end{aligned}
	\end{equation*}

	 We have too that $D(0)=B_{a(0)}^M=B_R^M$ and $D(T)=B_{a(T)}^M=\{o\}$, where $o$ is the center of the geodesic ball $B_R^M$.
	
	We consider the symmetrization in $\MW$ of the sets $\drm(t)=B^M_{a(t)}\subseteq\brm_R^M\subseteq M$, namely, the geodesic balls $\drm(t)^*=\brm_{\widetilde{r}(t)}^\omega(o_\omega)$ in $\MW$ such that
	
	\begin{equation*}
		\Vol\left(\drm(t)\right)=\Vol\left(\brm_{\widetilde{r}(t)}^\omega(o_\omega)\right)\, .
	\end{equation*}
	
	%Note that
	
	%\begin{equation}\label{eq:Dboundarysphere}
		%\partial \drm(t)^*=\partial\brm_{\widetilde{r}(t)}^\omega=\srm_{\widetilde{r}(t)}^\omega
	%\end{equation}
	
	%Let us consider now  the function $R:[0, T] \rightarrow  [0,s(R)]$ defined in Definition \ref{defR}. 
	
	For each $t \in [0,T]$, let us consider the function $\widetilde{r}(t)$, defined in Definition \ref{defR}. Then, in this particular context,  we have that $\widetilde{r}\,:\,[0,T]\longrightarrow[0,s(R)]$ is strictly decreasing and hence, bijective. In fact, note that if $t_1, t_2 \in [0,T]$ such that $t_1 < t_2$, then, as $a(t)$ is strictly decreasing, $a(t_1) > a(t_2)$, so	
	$$\Vol\left(\brm_{\widetilde{r}(t_1)}^\omega(o_\omega)\right)=\Vol\left(\brm_{a(t_1)}^M\right)>\Vol\left(\brm_{a(t_2)}^M\right)=\Vol\left(\brm_{\widetilde{r}(t_2)}^\omega(o_\omega)\right)$$
	\noindent  and hence $\widetilde{r}(t_1)> \widetilde{r}(t_2)$. 
	
	On the other hand, applying Lemma \ref{prop:Rfuncsymm}, we have that  for all $t \in R_\varphi=(0,T)$,
	\begin{equation}\label{eq3.11}
			\widetilde{r}{'}(t)=\dfrac{-1}{\Vol\left(\srm_{\widetilde{r}(t)}^\omega\right)}\int_{\Gamma(t)}\norm{\nabla^M\varphi}^{-1}\,d\sigma_t.
		\end{equation} 
		
		The inverse of $\widetilde{r}$ is the decreasing function
	
	\begin{equation*}
		\phi\,:\,[0,s(R)]\longrightarrow  [0,T];\quad \phi:=\phi(\widetilde{r}),
	\end{equation*}
	
	\noindent such that $\phi'\left(\widetilde{r}(t)\right)=\frac{1}{\widetilde{r}'(t)}$ for all $t\in  [0,T]$, $\phi(0)=T$ and $\phi(s(R))=0$.

	With all this background, we can say now that, in accordance with Definition \ref{def:simfunc} and Theorem \ref{prop:Propietatssymmobjects}, the symmetrization of  $\varphi=\E^\omega_R:\,\brm_R^M\longrightarrow\R$ is a radial function $\varphi^*=\E_R^{\omega^{*}}\,:\,\brm_{s(R)}^\omega\longrightarrow\R$ which satisfies the following equality	
	\begin{equation}\label{eq:Psi}
		\varphi^*(x^*)=\E_R^{\omega^{*}}(x^*)=\E_R^{\omega^{*}}\left(r_{o_\omega}(x^*)\right)
		=t_0=\phi\left(\widetilde{r}(t_0)\right)=\phi\left(\widetilde{r}\right).
	\end{equation}
	
	To see equation (\ref{eq:Psi}), we argue as follows: given $x^*\in B^M_R(o)^*=\brm_{s(R)}^\omega(o_\omega)=\cup_{t\in[0,T]}\srm_{\widetilde{r}(t)}^\omega(o_\omega)$, (concerning the second equality, recall that $\widetilde{r}: [0,T] \rightarrow [0,s(R)]$ is bijective), there exists some biggest value $t_0$ such that $r_{o_\omega}(x^*)=\widetilde{r}(t_0)$ and, hence, $x^*\in\brm_{\widetilde{r}(t_0)}^\omega =\drm(t_0)^*$. We then have that 
	\begin{equation}\label{eq:Psi2}
	\varphi^*(x^*)=\varphi^*(r_{o_{\omega}}(x^*))=\sup\{t \geq 0/ x^* \in B^w_{\widetilde{r}(t)}(o)\}=t_0=\phi\left(\widetilde{r}(t_0)\right)
	\end{equation}

	\noindent and hence, for all  $t\in(0,T)$, $\varphi^*\equiv \varphi^*(\widetilde{r}(t))$ and we have, applying equation \eqref{eq3.11}:
	
	\begin{equation}\label{eq:1}
	\begin{aligned}
		\frac{d}{d\widetilde{r}}\vert_{\widetilde{r}=\widetilde{r}(t)} \varphi^*(\widetilde{r})=\varphi^{*}{'}(\widetilde{r}(t))&=\E_R^{\omega^{*}}{'}(\widetilde{r}(t))=\phi{'}(\widetilde{r}(t))\\&=\dfrac{1}{\widetilde{r}{'}(t)}=-\frac{\Vol\left(\srm_{\widetilde{r}(t)}^\omega\right)}{ \int_{\Gamma(t)}\norm{\nabla^M\varphi}^{-1}\,d\sigma_t }.
	\end{aligned}
	\end{equation}

	But, as $\Vert \nabla^M \varphi(x)\Vert=\vert \psi'(r_o(x))\vert \neq 0$ for all $x \in B^M_R(o) -\{o\}$ and  $\Gamma(t)=\srm_{a(t)}^M$ for all $t\in R_\varphi = (0,T)$,  we conclude that
	
	\begin{equation}\label{eq4.2}
		%\begin{split}
			\int_{\Gamma(t)}\norm{\nabla^M\varphi}^{-1}\,d\sigma_t  =  \dfrac{1}{\vert\psi{'}(a(t))\vert}\Vol\left(\srm_{a(t)}^M\right)
		%\end{split}
	\end{equation}
	
	\noindent and hence, equation (\ref{eq:1}) becomes, using equation (\ref{eq4.2}), and the fact that $\psi=E^\omega_R$, in the following expression, for all $t \in [0,T]$:
	
	\begin{equation}\label{eq:psisymmprimavolums}
		\begin{split}
			\varphi^*{'}(\widetilde{r}(t)) & =-\modul{\psi{'}(a(t))}\dfrac{\Vol\left(\srm_{\widetilde{r}(t)}^\omega\right)}{\Vol\left(\srm_{a(t)}^M\right)}\\
			& = -\dfrac{\Vol\left(\brm_{a(t)}^\omega\right)}{\Vol\left(\srm_{a(t)}^\omega\right)}\,\dfrac{\Vol\left(\srm_{\widetilde{r}(t)}^\omega\right)}{\Vol\left(\srm_{a(t)}^M\right)}.
		\end{split}
	\end{equation}
	
	On the other hand, let us assume that $\hrm_{\srm_r^\omega}\leq\,\hrm_{\srm_r^M}$, for all $r \in (0,R]$. Then by Corollary \ref{cor:MeanComp} we know that  $\Vol\left(\brm_r^\omega\right)\,\leq\,\Vol\left(\brm_r^M\right)$ for all $r \in [0,R]$. Therefore,  
	
	\begin{equation}\label{eq4.14}
		\Vol\left(\brm_{\widetilde{r}(t)}^\omega\right)=\Vol\left(\brm_{a(t)}^M\right)\,\geq\,\Vol\left(\brm_{a(t)}^\omega\right), \quad\text{for all }t\in[0,T].
	\end{equation}
	
	Then, since $\Vol\left(\brm_r^\omega\right)$ is an increasing function, because $\frac{d}{dr}\Vol\left(\brm_r^\omega\right)=\Vol\left(\srm_r^\omega\right)\geq 0$, we have that 
	
	\begin{equation}\label{eq4.151}
		\widetilde{r}(t)\,\geq\,a(t),\quad\text{for all }t\in[0,T].
	\end{equation}
	
	\noindent so, since $\MW$ is balanced from above,  $q_\omega{'}(r)\geq 0$, we obtain:
	
	\begin{equation}\label{eq4.15}
		\dfrac{\Vol\left(\brm_{\widetilde{r}(t)}^\omega\right)}{\Vol\left(\srm_{\widetilde{r}(t)}^\omega\right)}\,\geq\,\dfrac{\Vol\left(\brm_{a(t)}^\omega\right)}{\Vol\left(\srm_{a(t)}^\omega\right)}, \quad\text{for all }t\in[0,T].
	\end{equation}
	
	Therefore, using equation \eqref{eq:psisymmprimavolums} and the fact that $\Vol\left(\brm_{\widetilde{r}(t)}^\omega \right)=\Vol\left(\brm_{a(t)}^M\right)$, we have
	
	\begin{equation}\label{eq.4.16}
		\begin{split}
			\E_R^{\omega *}{'}(\widetilde{r}(t))=\varphi^*{'}(\widetilde{r}(t)) \geq\,& -\dfrac{\Vol\left(\brm_{a(t)}^M\right)}{\Vol\left(\srm_{a(t)}^M\right)},\quad\text{for all }t\in[0,T].
		\end{split}
	\end{equation}
	
	Now, we apply  Proposition \ref{prop:MeanTorEqualities}, the isoperimetric inequality (\ref{cor:MeanComp1}) of Corollary \ref{cor:MeanComp}, the fact that $\widetilde{r}(t)\geq\,a(t)$, and that $q_\omega{'}\geq 0$, to obtain finally
	
	\begin{equation}\label{eq4.17}
		%\begin{split}
			\E_R^{\omega *}{'}(\widetilde{r}(t)) \geq\,-\dfrac{\Vol\left(\brm_{a(t)}^M\right)}{\Vol\left(\srm_{a(t)}^M\right)}
			 \geq\,-\dfrac{\Vol\left(\brm_{\widetilde{r}(t)}^\omega\right)}{\Vol\left(\srm_{\widetilde{r}(t)}^\omega\right)}=E_{s(R)}^\omega{'}(\widetilde{r}(t))\,\,\forall t \in (0,T).
		%\end{split}
	\end{equation}
	
	Now, as  $\E_R^{\omega *}{'}(\widetilde{r}) \geq\,E_{s(R)}^\omega{'}(\widetilde{r})\,\,\forall \widetilde{r} \in (0,s(R))$, we have, integrating along $[0,s(R)]$, and taking into account that $\E_R^{\omega *}(s(R))=E_{s(R)}^\omega(s(R))=0$,
	\begin{equation}
	\begin{aligned}
	-\E_R^{\omega *}(\widetilde{r})&=\int_{\widetilde{r}}^{s(R)} \E_R^{\omega *}{'}(u)du \geq \\&
	\int_{\widetilde{r}}^{s(R)} E_{s(R)}^\omega{'}(u)du=-E_{s(R)}^\omega(\widetilde{r})\,\,\,\forall \widetilde{r} \in [0,s(R)]
	\end{aligned}
	\end{equation}
	\noindent so 
	$$\E_R^{\omega *}(\widetilde{r}) \leq\,E_{s(R)}^\omega(\widetilde{r})\,\,\,\forall \widetilde{r} \in [0,s(R)].$$
	
	If we assume that $\hrm_{\srm_r^\omega}\geq\,\hrm_{\srm_r^M}$, for all $r \in [0,R]$, we use the same argument, changing all the inequalities, to obtain
	\begin{equation}
	\begin{aligned}
	\E_R^{\omega *}{'}(\widetilde{r}(t)) &\leq\,E_{s(R)}^\omega{'}(\widetilde{r}(t))\,\,\forall t \in (0,T)\,\,\text{and hence}\\
	\E_R^{\omega *}(\widetilde{r}) &\geq\,E_{s(R)}^\omega(\widetilde{r})\,\,\,\forall \widetilde{r} \in [0,s(R)].
	\end{aligned}
	\end{equation}
	
	We are going to study the case of equality, when we assume the hypothesis $\hrm_{\srm_r^\omega}\leq\,\hrm_{\srm_r^M}$, for all $r \in [0,R]$, (the discussion of equality if we assume $\hrm_{\srm_r^\omega}\geq\,\hrm_{\srm_r^M}$, for all $r \in [0,R]$ is the same, mutatis mutandi).
	
	Equality $\E_R^{\omega *}(\widetilde{r})=E_{s(R)}^\omega(\widetilde{r}) \,\,\forall r \in ]0,s(R)]$ implies equality $\E_R^{\omega *}{'}(\widetilde{r}) =\,E_{s(R)}^\omega{'}(\widetilde{r})\,\,\forall \widetilde{r} \in (0,s(R))$, which in its turn implies that inequalities in (\ref{eq4.17}) and hence, in (\ref{eq.4.16}) and (\ref{eq4.15}) become equalities for all $t \in [0,T]$. In particular, from equality in (\ref{eq4.15}) and inequality (\ref{eq4.14}), we deduce that
	\begin{equation}\label{eq4.18}
		\dfrac{\Vol\left(\srm_{a(t)}^\omega\right)}{\Vol\left(\srm_{\widetilde{r}(t)}^\omega\right)}\,\leq\,1, \quad\text{for all }t\in[0,T].
	\end{equation}
	On the other hand, using again equality in inequality (\ref{eq4.15}) and having into account that, as we assume that $\hrm_{\srm_r^\omega}\leq\,\hrm_{\srm_r^M}$, for all $r \in [0,R]$, then we have  isoperimetric inequality (\ref{cor:MeanComp1}), we obtain:
	
	\begin{equation}\label{eq4.19}
		\dfrac{\Vol\left(\brm_{\widetilde{r}(t)}^\omega\right)}{\Vol\left(\srm_{\widetilde{r}(t)}^\omega\right)}\,=\,\dfrac{\Vol\left(\brm_{a(t)}^\omega\right)}{\Vol\left(\srm_{a(t)}^\omega\right)}\,\geq\,\dfrac{\Vol\left(\brm_{a(t)}^M\right)}{\Vol\left(\srm_{a(t)}^M\right)} , \quad\text{for all }t\in[0,T],
	\end{equation}	
	\noindent and hence, as $\Vol\left(\brm_{\widetilde{r}(t)}^\omega\right)=\Vol\left(\brm_{a(t)}^M\right)$ and using (\ref{eq4.18}):
	\begin{equation}\label{eq4.20}
	\Vol\left(\srm_{a(t)}^\omega\right) \leq \Vol\left(\srm_{\widetilde{r}(t)}^\omega\right) \leq \Vol\left(\srm_{a(t)}^M\right).
	\end{equation}

Now, differentiating the equality 
\begin{equation}\label{eq4.21}
		\Vol\left(\brm_{\widetilde{r}(t)}^\omega\right)=\Vol\left(\brm_{a(t)}^M\right),\quad\text{for all }t\in[0,T],
	\end{equation}
	\noindent we obtain
	\begin{equation}\label{eq4.22}
		\Vol\left(\srm_{\widetilde{r}(t)}^\omega\right)\widetilde{r'}(t)=\Vol\left(\srm_{a(t)}^M\right)a'(t),\quad\text{for all }t\in(0,T),
	\end{equation}
	\noindent and hence, using inequality (\ref{eq4.20}),
	
	\begin{equation}\label{eq4.23}
		\dfrac{\widetilde{r'}(t)}{a'(t)}\,=\,\dfrac{\Vol\left(\srm_{a(t)}^M\right)}{\Vol\left(\srm_{\widetilde{r}(t)}^\omega\right)}\,\geq\,1, \quad\text{for all }t\in(0,T),
	\end{equation}	
	\noindent so $\widetilde{r'}(t) \geq a'(t)\,\,\forall t \in (0,T)$, and therefore, as $\widetilde{r}(T)=a(T)=0$, we finally obtain, integrating along $[0,T]$, that $\widetilde{r}(t)\leq a(t)\,\,\forall t \in [0,T]$. Hence, as we know, (see inequality (\ref{eq4.151})), that $\widetilde{r}(t)\geq a(t)\,\,\forall t \in [0,T]$, we obtain
	$$\widetilde{r}(t)= a(t)\,\,\forall t \in [0,T].$$
	
	Therefore, $s(R)=\widetilde{r}(0)= a(0)=R$\, and, moreover, $\Vol\left(\brm_{\widetilde{r}(t)}^\omega \right)=\Vol\left(\brm_{\widetilde{r}(t)}^M\right)\,\,\text{for all }t\in[0,T]$, so
	$$\Vol\left(\brm_{r}^\omega \right)=\Vol\left(\brm_{r}^M\right)\,\,\forall r \in [0,R]$$
	\noindent and hence
	$$\Vol\left(\srm_{r}^\omega \right)=\Vol\left(\srm_{r}^M\right)\,\,\forall r \in [0,R].$$
	
	Moreover, we apply the equality assertion in Corollary \ref{cor:MeanComp} to conclude that $\hrm_{\srm_r^\omega}=\,\hrm_{\srm_r^M},\;\,\text{for all }0\leq r \leq R$

\end{proof}

As a consequence of the Proposition \ref{prop:ineqpsiE} we have the following result, where it is proved that, under our hypotheses, the torsional rigidity of the geodesic balls determines its first Dirichlet eigenvalue:

\begin{theorem}\label{teo:2.1}
	Let $(M^n,g)$ be a complete Riemannian manifold and let $(\MW,g_\omega)$ be a rotationally symmetric model space with center $o_\omega \in \MW$, balanced from above. Let $o \in M$ be a point in $M$ and let us suppose that  $inj(o) \leq inj(o_\omega)$. Let us consider a metric ball $B^M_R(o)$, with $R < inj(o) \leq inj(o_w)$. Let us suppose moreover that the mean curvatures of the geodesic spheres in $M$ and $M_{\omega}$ satisfies

	\begin{equation}\label{eq:meancurvatureconditions9}
		\hrm_{\srm_r^\omega}\leq\,(\geq)\, \hrm_{\srm_r^M}\quad\text{for all}\quad 0 < r \leq R.
	\end{equation}
	
	Then
	
	\begin{equation}\label{torsrigcomp}
		\CA_1\left(\brm_{s(R)}^\omega\right)\geq\,(\leq)\,\CA_1\left(\brm_R^M\right)
	\end{equation}
	
	\noindent where $\brm_{s(R)}^\omega$ is the Schwarz symmetrization of $\brm_R^M$ in the model space $(\MW,g_\omega)$.

	Equality in any of  inequalities (\ref{torsrigcomp})  implies the equality among the radius $s(R)=R$ and that 
	 $$\hrm_{\srm_R^\omega(o_\omega)}=\,\hrm_{\srm_R^M(o)}\,\,\,\text{for all }0<r\leq R$$
	 \noindent and hence, we have the equalities
\begin{enumerate}
	\item Equality $\bu_{k,R}^{\omega}=u_{k,R}$ on $B^M_R(o)$ for all $k\geq 1$, and hence, equality $\bu_{k,r}^{\omega}=u_{k,r}$ on $B^M_r(o)$ for all $k\geq 1$ and for all $0 < r \leq R$.
	\item Equalities $\Vol\left(\brm_r^\omega(o_\omega)\right)=\Vol\left(\brm_r^M(o)\right)$ and $\Vol\left(\srm_r^\omega(o_\omega)\right)=\Vol\left(\srm_r^M(o)\right)\,\,\,\text{for all }\, 0<r \leq R$.
	\item Equalities $\CA_k\left(\brm_r^\omega(o_\omega)\right)=\CA_k\left(\brm_r^M(o)\right)$,  for all $k\geq 1$ and for all $0<r \leq R$.
	\item Equality $\lambda_1(B^w_r(o_\omega))=\lambda_1(B^M_r(o))$ for all $0<r \leq R$.
	
	\noindent Namely, the Torsional Rigidity determines the Poisson hierarchy, the volume, the $L^1$-moment spectrum and the first Dirichlet eigenvalue of the ball $B^M_r(o)$ for all $0<r \leq R$.

	\end{enumerate}
	
\end{theorem}
 
\begin{proof}

Let us consider a symmetric model space rearrangement of the metric ball $\brm_R^M$ as it has been described in Definition \ref{def:Symm} and Definition \ref{def:simfunc}, namely, a symmetrization of $\brm_R^M$ which is a geodesic $s(R)$-ball in the model space $\MW$ such that $\Vol\left(\brm_R^M\right)=\Vol\left(\brm_{s(R)}^\omega\right)$, together the symmetrization $\E_R^\omega{^*}\,:\,\brm_{s(R)}^\omega\longrightarrow\R$ of the transplanted mean exit time function $\E_R^\omega\,:\,\brm_R^M\longrightarrow\R$.

%Given the solution $E_R^M$ of the Poisson Problem \eqref{eq:moments1} on $\brm_R^M$, we compare it with the transplanted mean exit time function $\E_R^\omega(x)=E_R^\omega(r_o(x))$, namely, with the solution of the Poisson Problem \eqref{eq:moments1}, defined on the geodesic $R$-ball $\brm_R^\omega \subseteq \MW$, transplanted into the ball $\brm_R^M$.
	
	Applying Theorems \ref{teo:MeanComp} and \ref{cor:eqschwarz}  and Proposition \ref{prop:ineqpsiE}, we have that
	
	%\begin{equation*}
		%E_R^M \leq\,(\geq)\,\E_R^\omega \quad\text{on }\brm_R^M
	%\end{equation*}
	
	%Using this inequality and Corollary \ref{cor:eqschwarz}, we have
	
	\begin{equation}\label{eq:4}
	\begin{aligned}
		\CA_1\left(\brm_R^M\right)&=\int_{\brm_R^M}E_R^M\,d\sigma\leq\,(\geq)\,\int_{\brm_R^M}\E_R^\omega\,d\sigma\\&=\int_{\brm_{s(R)}^\omega}\E_R^\omega{^*}\,d\widetilde{\sigma} \leq\,(\geq)\,\int_{\brm_{s(R)}^\omega}E_{s(R)}^\omega\,d\widetilde{\sigma}=\CA_1\left(\brm_{s(R)}^\omega\right)
		\end{aligned}
	\end{equation}
	
	\noindent and the Theorem is proved.

	We are going to study the case of equality, when we assume the hypothesis $\hrm_{\srm_r^\omega}\leq\,\hrm_{\srm_r^M}$, for all $r \in [0,R]$, (the discussion of equality if we assume $\hrm_{\srm_r^\omega}\geq\,\hrm_{\srm_r^M}$, for all $r \in [0,R]$ is the same, mutatis mutandi).
	
	Equality in (\ref{eq:4}) implies that all the inequalities contained in this expression become equalities. In particular, we have that $\int_{\brm_R^M}E_R^M\,d\sigma=\int_{\brm_R^M}\E_R^\omega\,d\sigma$ and that $\int_{\brm_{s(R)}^\omega}\E_R^\omega{^*}=\int_{\brm_{s(R)}^\omega}E_{s(R)}^\omega\,d\widetilde{\sigma}$. 
	
	From this second equality and inequality (\ref{eq:ineqpsiE2}) in Proposition \ref{prop:ineqpsiE},  we have that $\E_R^\omega{^*}=E_{s(R)}^\omega$ on $[0,s(R)]$. Applying again Proposition \ref{prop:ineqpsiE}, we deduce that $s(R)=R$, and that $\hrm_{\srm_r^\omega}=\,\hrm_{\srm_r^M}\,\,\,\text{for all }0<r \leq R$.  
	
	On the other hand, equality $\int_{\brm_R^M}E_R^M\,d\sigma=\int_{\brm_R^M}\E_R^\omega\,d\sigma$ implies, using Theorem \ref{teo:MeanComp}, that $E_R^M=\E_R^\omega$ on $\brm_R^M$. Hence we conclude equality $\bu_{k,R}^{\omega}=u_{k,R}$ on $B^M_R(o)$ for all $k\geq 1$ using assertion (3) in Theorem \ref{teo:TowMomComp} and that $\bu_{k,r}^{\omega}=u_{k,r}$ on $B^M_r(o)$ for all $k\geq 1$ and for all $r \in [0,R]$ using assertion (4) in Theorem \ref{teo:TowMomComp}.

	Moreover, equality $E_R^M=\E_R^\omega$ on $\brm_R^M$ implies, using equality conclusions in Corollary \ref{cor:MeanComp2}, that, for all  $r \in ]0,R]$,
	
	\begin{equation}
	\begin{aligned}		
\dfrac{\Vol\left(\brm_r^\omega(o_\omega)\right)}{\Vol\left(\srm_r^\omega(o_\omega)\right)\,}&=\,\dfrac{\Vol\left(\brm_r^M(o)\right)}{\Vol\left(\srm_r^M(o)\right)},\\
\Vol\left(\brm_r^\omega\right)\, &=\,\Vol\left(\brm_r^M\right),\\
	\Vol(S^\omega_r) &= \Vol(S^M_r).
	\end{aligned}
	\end{equation}
	
	Hence, as we are assuming that $\CA_1\left(\brm_R^M\right)=\CA_1\left(\brm_{s(R)}^\omega\right)$ and we have deduced $s(R)=R$, then we obtain the equality $$\dfrac{\CA_1\left(\brm_R^\omega\right)}{\Vol\left(\srm_R^\omega\right)}=\dfrac{\CA_1\left(\brm_R^M\right)}{\Vol\left(\srm_R^M\right)}$$ and hence, by Corollary \ref{isoptors}, $\CA_k\left(\brm_r^M\right)=\CA_k\left(\brm_{r}^\omega\right)$ for all $k \geq 1$ and for all $0 <r \leq R$.
	
	Finally, as equality $\CA_1\left(\brm_R^M\right)=\CA_1\left(\brm_{s(R)}^\omega\right)$ implies equalities  $\CA_k\left(\brm_r^\omega\right)=\CA_k\left(\brm_r^M\right)$,  for all $k\geq 1$ and for all $r \in ]0,R]$, we have that, given $B^M_r \subseteq M$ in a Riemannian manifold $(M,g)$ with $r \in ]0,R]$, (see \cite{HuMP3} and \cite{BGJ}):
\begin{equation}\label{gim}
\begin{aligned}
\lambda_1(B^M_r)&=\lim_{k \to \infty}\dfrac{k\CA_{k-1}\left(\brm_r^M\right)}{\CA_{k}\left(\brm_r^M\right)}\\&=\lim_{k \to \infty}\dfrac{k\CA_{k-1}\left(\brm_r^\omega\right)}{\CA_{k}\left(\brm_r^\omega\right)}=\lambda_1(B^\omega_r).
\end{aligned}
\end{equation}

\end{proof}

\begin{corollary}\label{cor:2.1}
	Let $(M^n,g)$ be a complete Riemannian manifold and let $(\MW,g_\omega)$ be a rotationally symmetric model space with center $o_\omega \in \MW$, balanced from above. Let $o \in M$ be a point in $M$ and let us suppose that  $inj(o) \leq inj(o_\omega)$. Let us consider a metric ball $B^M_R(o)$, with $R < inj(o) \leq inj(o_w)$. Let us suppose moreover that the mean curvatures of the geodesic spheres in $M$ and $M_{\omega}$ satisfies

	\begin{equation}\label{eq:meancurvatureconditions10}
		\hrm_{\srm_r^\omega}\leq\,(\geq)\, \hrm_{\srm_r^M}\quad\text{for all}\quad 0 < r \leq R.
	\end{equation}
	Then
	
	\begin{equation}\label{torsrigcomp2}
		\CA_1\left(\brm_R^M\right)\,\leq E_{s(R)}^\omega(0) \Vol(\brm_{R}^M).
	\end{equation}
	\end{corollary}
	\begin{proof}
	Assuming that $\hrm_{\srm_r^\omega}\leq\ \hrm_{\srm_r^M}\,\,\text{for all}\,\, 0 < r \leq R$, we use equation (\ref{eq:4}) to obtain, having into account that $E_{s(R)}^\omega(r) \leq E_{s(R)}^\omega(0)\,\forall r \in ]0,s(R)]$, 
	
		\begin{equation}
	\begin{aligned}
		\CA_1\left(\brm_R^M\right)&=\int_{\brm_R^M}E_R^M\,d\sigma\, \leq \,\int_{\brm_R^M}\E_R^\omega\,d\sigma\\&=\int_{\brm_{s(R)}^\omega}\E_R^\omega{^*}d\widetilde{\sigma}\leq\,\int_{\brm_{s(R)}^\omega}E_{s(R)}^\omega\,d\widetilde{\sigma}\leq E_{s(R)}^\omega(0) \Vol(\brm_{R}^M).
		\end{aligned}
	\end{equation}
	\end{proof}
	
	\begin{remark}
	As $(\MW,g_\omega)$ is balanced from above, then $\dfrac{d}{dr}\left(q_\omega(r)\right)   \geq 0$, so $q_\omega(r)$ is non-decreasing with $r$. Then, as $E_{s(R)}^\omega(r(x))=\psi(r(x))=\int_{r(x)}^{s(R)} q_\omega(t)\,dt$, we have that 
	$$E_{s(R)}^\omega(0)=\int_{0}^{s(R)} q_\omega(t)\,dt \leq s(R) q_\omega(s(R))=s(R) \dfrac{\Vol\left(\brm_{s(R)}^\omega(o_\omega)\right)}{\Vol\left(\srm_{s(R)}^\omega(o_\omega)\right)}$$
	\noindent so 
	$$\CA_1\left(\brm_R^M\right)\,\leq \E_{s(R)}^\omega(0) \Vol(\brm_{R}^M) \leq s(R) \dfrac{\Vol\left(\brm_{s(R)}^\omega(o_\omega)\right)}{\Vol\left(\srm_{s(R)}^\omega(o_\omega)\right)} \Vol(\brm_{R}^M).$$
	
	\end{remark}

%%%%%	%%%%%%%%%%%%%%%%%%%%%%%%%%%%%%%%
\section{ Cheng's first Dirichlet eigenvalue comparison and the determination of the moment spectrum of a geodesic ball }\label{sec:TorRidCom2}\
%%%%%%%%%%%%%%%%%%%%%%%%%%%%%%%%%%

%%%%%%%%%%%%%%%%%%%%%%%%%%%%%%%%%%%%%%%
%%%%%%%%%%%%%%%%%%%%%%%%%%%%%%%%%%%%%%%%
%\subsection{Torsional rigidity and the First Dirichlet Eigenvalue comparison}\label{sec:MomentsComp}\
%%%%%%%%%%%%%%%%%%%%%%%%%%%%%%%%%%%%%%%
%%%%%%%%%%%%%%%%%%%%%%%%%%%%%%%%%%%%%%

\medskip 
Finally, as a corollary of the previous results, we have in Theorem \ref{th_const_below2} a Cheng's first Dirichlet eigenvalue comparison, (see \cite{BM}). On the other hand, in Corollary \ref{th_const_below3}, we have been able to show that, under our hypotheses,  the first Dirichlet eigenvalue of geodesic balls determines its exit time moment spectrum and its Poisson hierarchy.

\begin{theorem}\label{th_const_below2}
Let $(M^n,g)$ be a complete Riemannian manifold and let $(\MW,g_\omega)$ be a rotationally symmetric model space with center $o_\omega \in \MW$. Let $o \in M$ be a point in $M$ and let us suppose that  $inj(o) \leq inj(o_\omega)$. Let us consider a metric ball $B^M_R(o)$, with $R < inj(o) \leq inj(o_w)$. Let us suppose moreover that the mean curvatures of the geodesic spheres in $M$ and $M_{\omega}$ satisfies

	\begin{equation}\label{eq:meancurvatureconditions12}
		\hrm_{\srm_r^\omega}\leq\,(\geq)\, \hrm_{\srm_r^M}\quad\text{for all}\quad 0 < r \leq R.
	\end{equation}
	
	Then we have the inequalities
 \begin{equation}\label{ineqleq_submanifold2}
  \lambda_1 (B^\omega_R) \leq\,(\geq)\,\lambda_1(B^M_R)\end{equation}
\noindent where $B^\omega_R$ is the geodesic ball in $M^m_\omega$.
 
\noindent  Equality in any of these inequalities implies that 
$$\hrm_{\srm_r^\omega}=\,\hrm_{\srm_r^M}\,\,\,\text{for all }0<r\leq R$$
\noindent and hence, we have the equalities

\begin{enumerate}
	\item Equality $\bu_{k,R}^{\omega}=u_{k,R}$ on $B^M_R(o)$ for all $k\geq 1$, and hence, equality $\bu_{k,r}^{\omega}=u_{k,r}$ on $B^M_r(o)$ for all $k\geq 1$ and for all $0 < r \leq R$.
	\item Equalities $\Vol\left(\brm_r^\omega\right)=\Vol\left(\brm_r^M\right)$ and $\Vol\left(\srm_r^\omega\right)=\Vol\left(\srm_r^M\right)\,\,\,\text{for all }\, 0<r \leq R$.
	\item Equalities $\CA_k\left(\brm_r^\omega\right)=\CA_k\left(\brm_r^M\right)$,  for all $k\geq 1$ and for all $0<r \leq R$.
	
	\noindent Namely, the first Dirichlet eigenvalue determines  the Poisson hierarchy, the volume, and  the $L^1$-moment spectrum of the ball $B^M_r(o)$ for all $0<r \leq R$.

	\end{enumerate}

  \end{theorem}

\begin{proof}

The proof  follows the lines of the proof of Theorems 6 and 7 in \cite{HuMP3}. This technique is based in the the  description of the first Dirichlet eigenvalue of a smooth precompact domain $D$ in a Riemannian manifold
given by P. McDonald and R. Meyers in \cite{McMe}. 

When $D=B^M_R$, we have
\begin{equation}\label{desc}
\lambda_1(B^M_R)= \sup \left\{\eta \geq 0 \,:\, \lim_{k \to \infty}
\sup
\left(\frac{\eta}{2}\right)^k\frac{\mathcal{A}_{k}(B^M_R)}{\Gamma(k+1)}
<\infty\right\} \quad .
\end{equation}

Let us assume first that $\hrm_{\srm_r^\omega}\leq\,\hrm_{\srm_r^M},\,\,\text{for all }0<r\leq R$. Then, we have, by Corollary \ref{isoptors}, that

\begin{equation}\label{ineqspec2}
%\begin{aligned}
\frac{\mathcal{A}_{k}(B^M_R)}{\Vol(S^M_R)} \leq
\frac{\mathcal{A}_{k}(B^\omega_{R})}{\Vol(S^\omega_{R})}\,\,\,\,\textrm{for all}\,\, k \in \ene.
%\end{aligned}
\end{equation}

On the other hand, by Corollary \ref{cor:MeanComp}:

\begin{equation}\label{meyer1}
\frac{\Vol(S^M_R)}{\Vol(S^\omega_{R})} \geq \frac{\Vol(B^M_R)}{\Vol(B^\omega_{R})}\geq 1.
\end{equation}

Then, using inequality (\ref{ineqspec2}) the set $$\mathcal{F}_2:=\{\eta \geq 0 \,:\, \lim_{k \to \infty}  \sup \left(\frac{\eta}{2}\right)^k\frac{\mathcal{A}_{k}(B^\omega_{R})}{\Gamma(k+1)} \frac{\Vol(S^M_R)}{\Vol(S^\omega_{R})}<\infty\}$$  is included in the set $$\mathcal{F}_1:= \{\eta \geq 0 \,:\, \lim_{k \to \infty} \sup \left(\frac{\eta}{2}\right)^k\frac{\mathcal{A}_{k}(B^M_R)}{\Gamma(k+1)} <\infty\}\quad ,$$

\noindent so we have, using this last observation and  inequality (\ref{meyer1}),

\begin{equation}\label{meyer2}
\begin{aligned}
\lambda_1(B^M_R) & = \sup \{\eta \geq 0 \,:\, \lim_{k \to \infty}  \sup \left(\frac{\eta}{2}\right)^k\frac{\mathcal{A}_{k}(B^M_R)}{\Gamma(k+1)} <\infty\} \\ 
& \geq
\sup \{\eta \geq 0 \,:\, \lim_{k \to \infty}  \sup \left(\frac{\eta}{2}\right)^k\frac{\mathcal{A}_{k}(B^\omega_{R})}{\Gamma(k+1)} \frac{\Vol(S^M_R)}{\Vol(S^\omega_{R})}<\infty\} \\ 
& = 
\frac{\Vol(S^M_R)}{\Vol(S^\omega_{R})}\sup \{\eta \geq 0 \,:\, \lim_{k \to \infty} \sup \left(\frac{\eta}{2}\right)^k\frac{\mathcal{A}_{k}(B^\omega_{R})}{\Gamma(k+1)} <\infty\}\\
& =
\frac{\Vol(S^M_R)}{\Vol(S^\omega_{R})} \lambda_1(B^\omega_{R})\geq \lambda_1(B^\omega_{R}).
\end{aligned}
\end{equation}

If we assume $\hrm_{\srm_r^\omega}\geq\,\hrm_{\srm_r^M},\,\text{for all }0<r\leq R$, then we obtain $\lambda_1 (B^\omega_R) \geq \lambda_1(B^M_R)$ with the same argument, inverting all the inequalities.

Finally, equality $\lambda_1 (B^\omega_{R})=\lambda_1(B^M_{R})$ implies that all the inequalities in (\ref{meyer2}) are equalities, so  we have the equality in the inequality (\ref{meyer1}), (namely, the equality in the isoperimetric inequality (\ref{cor:MeanComp1}) in Corollary \ref{cor:MeanComp}), and moreover the equality between the volumes $  \Vol( B^\omega_{R})=\Vol(B^M_{R})$ and $\Vol\left(\srm_R^\omega\right)=\Vol\left(\srm_R^M\right)$. Hence, we have, by Corollary  \ref{cor:MeanComp}, the equalities 
 $$\hrm_{\srm_r^\omega}=\,\hrm_{\srm_r^M}\,\,\,\text{for all }0<r\leq R$$	
 
 \noindent and, in its turn, equalities $ \Vol( B^\omega_{r}) =\Vol(B^M_{r})$ and $\Vol\left(\srm_r^\omega\right)=\Vol\left(\srm_r^M\right)\,\,\forall r \in ]0,R]$. Assertions (1) and (2) follows from Proposition \ref{prop3.2} and Theorem \ref{teo:TowMomComp}.
\end{proof}

We finish the paper with a consequence of Theorems \ref{th_const_below2} and \ref{teo:TowMomComp} which summarizes the relation between the first Dirichlet eigenvalue, the $L^1$-moment spectrum and the Poisson hierarchy of the geodesic balls $B^M_R(o)$ of a Riemannian manifold which satisfies our restriction on the mean curvatures of the geodesic spheres included in it, $S^M_r(o)$, $r \leq R$.

\begin{corollary}\label{th_const_below3}
Let $(M^n,g)$ be a complete Riemannian manifold and let $(\MW,g_\omega)$ be a rotationally symmetric model space with center $o_\omega \in \MW$. Let $o \in M$ be a point in $M$ and let us suppose that  $inj(o) \leq inj(o_\omega)$. Let us consider a metric ball $B^M_R(o)$, with $R < inj(o) \leq inj(o_w)$. Let us suppose moreover that the mean curvatures of the geodesic spheres in $M$ and $M_{\omega}$ satisfies

	\begin{equation}\label{eq:meancurvatureconditions13}
		\hrm_{\srm_r^\omega}\leq\,(\geq)\, \hrm_{\srm_r^M}\quad\text{for all}\quad 0 < r \leq R.
	\end{equation}
	
	Then, the following equalities are equivalent:
  
  \begin{enumerate}
  \item \label{cor:5.2.1} $ \lambda_1 (B^\omega_R)=\lambda_1(B^M_R)$.
  \item \label{cor:5.2.2} $\CA_k(B^\omega_R)=\CA_k(B^M_R)\,\,\forall k \geq 1$.
  \item \label{cor:5.2.3} $\bu^w_{k,R}=u_{k,R} \,\,\forall k \geq 1$ in $B^M_R$.
  \end{enumerate}
  
  Moreover, equality $\hrm_{\srm_r^\omega}= \hrm_{\srm_r^M}\,\,\, \text{for all}\, \,\,0 < r \leq R$ implies any, (and hence, all), of the equalities $(1)$, $(2)$ and $(3)$.
 \end{corollary}

\begin{proof}

%Finally, equality $\lambda_1(B^M_{R}) = \lambda_1 (B^\omega_{R})$ implies that all the inequalities in (\ref{meyer2}) are equalities, so  we have the equality in the inequality (\ref{meyer1}), (namely, the equality in the isoperimetric inequality (\ref{cor:MeanComp1}) in Corollary \ref{cor:MeanComp}), and moreover the equality between the volumes $ \Vol(B^M_{R}) = \Vol( B^\omega_{R})$ and $\Vol\left(\brm_R^\omega\right)=\Vol\left(\brm_R^M\right)$. Hence, we have, by Corollary  \ref{cor:MeanComp}, the equalities 
 %$$\hrm_{\srm_r^\omega}=\,\hrm_{\srm_r^M}\,\,\,\text{for all }0<r\leq R$$

We are going to prove these equivalences. We first assume that  
$$\hrm_{\srm_r^\omega}\leq \hrm_{\srm_r^M}\quad\text{for all}\quad 0 < r \leq R.$$

We see first that equality \eqref{cor:5.2.1} implies equalities \eqref{cor:5.2.3}, namely, that the first Dirichlet eigenvalue of  the geodesic ball $B^M_R$ determines its Poisson hierarchy. To do that, we start with the last observation in Theorem \ref{th_const_below2}, namely, that equality $\lambda_1(B^\omega_{R}) = \lambda_1 (B^M_{R})$ implies that all the inequalities in (\ref{meyer2}) are equalities, so  we have the equality in the inequality (\ref{meyer1}), (namely, the equality in the isoperimetric inequality (\ref{cor:MeanComp1}) in Corollary \ref{cor:MeanComp}), and moreover the equality between the volumes $ \Vol(B^\omega_{r}) = \Vol( B^M_{r})$ and $\Vol\left(\srm_r^\omega\right)=\Vol\left(\srm_r^M\right)$ for all $r \in [0,R]$. Hence, we have, by Corollary  \ref{cor:MeanComp}, the equalities 
 $$\hrm_{\srm_r^\omega}=\,\hrm_{\srm_r^M}\,\,\,\text{for all }0<r\leq R.$$	

Then, by Proposition \ref{prop3.2}, we have that $\bu^w_{1,R}=u_{1,R}$ on $B^M_R$. Hence we conclude equality $\bu_{k,R}^{\omega}=u_{k,R}$ on $B^M_R(o)$ for all $k\geq 1$ using assertion (3) in Theorem \ref{teo:TowMomComp}. We have concluded that \eqref{cor:5.2.1} implies \eqref{cor:5.2.3}.

To see that equality \eqref{cor:5.2.1} implies equalities \eqref{cor:5.2.2}, we compute now as in Corollary \ref{isoptors}: for all $k \geq 1$, we have that, as $\Vol\left(\srm_r^\omega\right)=\Vol\left(\srm_r^M\right)$ for all $r \in [0,R]$,
	
	\begin{equation}
		\begin{split}
			\CA_k\left(\brm_R^M\right) & =\int_{\brm_R^M}u_{k,R}\,d\sigma=-\dfrac{1}{k+1}\int_{\brm_R^M}\Delta^M u_{k+1,R}d\sigma
			 \\&=\,-\dfrac{1}{k+1}\int_{\brm_R^M}\Delta^M\bu_{k+1,R}^{\omega}\,d\sigma= -\dfrac{1}{k+1}\,\bu_{k+1,R}^{\omega}{'}(R)\Vol\left(\srm_R^M\right)\\ &=\dfrac{\CA_k\left(\brm_R^\omega\right)}{\Vol\left(\srm_R^\omega\right)}\,\Vol\left(\srm_R^M\right)= \CA_k\left(\brm_R^\omega\right)
		\end{split}
	\end{equation}
	\noindent and hence we have equalities \eqref{cor:5.2.2}.
	
	To see that equalities \eqref{cor:5.2.2} implies equality \eqref{cor:5.2.1}, i.e., that the exit time moment spectrum of $B^M_R$ determines its first Dirichlet eigenvalue, we compute, using Theorem A in \cite{HuMP3} and as $\CA_k\left(\brm_R^M\right)= \CA_k\left(\brm_R^\omega\right)\,\forall k \geq 1$:
\begin{equation}\label{gim2}
\begin{aligned}
\lambda_1(B^M_R)&=\lim_{k \to \infty}\dfrac{k\CA_{k-1}\left(\brm_R^M\right)}{\CA_{k}\left(\brm_R^M\right)}\\&=\lim_{k \to \infty}\dfrac{k\CA_{k-1}\left(\brm_R^\omega\right)}{\CA_{k}\left(\brm_R^\omega\right)}=\lambda_1(B^\omega_R).
\end{aligned}
\end{equation}

To see that equalities \eqref{cor:5.2.3} implies equality \eqref{cor:5.2.1}, namely, that the Poisson hierarchy of the ball $B^M_R$ determines its first Dirichlet eigenvalue, we will see first that equalities \eqref{cor:5.2.3} implies equalities \eqref{cor:5.2.2}. Assuming that \eqref{cor:5.2.3} is satisfied, we have that $\bu_{k,R}^{\omega}=u_{k,R}$ on $B^M_R(o)$ for all $k\geq 1$. In particular, $\bu_{1,R}^{\omega}=u_{1,R}$ on $B^M_R(o)$, so, by Corollary \ref{cor:MeanComp2}, 
	$\Vol\left(\srm_r^\omega\right)=\Vol\left(\srm_r^M\right)$ and $\Vol\left(\brm_r^\omega\right)=\Vol\left(\brm_r^M\right)$ for all $r \in ]0, R]$ and, hence, given $r \in ]0,R]$, and for all $k \geq 1$,
	
	\begin{equation}\label{ineqisopTors2}
		\begin{split}
			\CA_k\left(\brm_R^M\right) & =\int_{\brm_R^M}u_{k,R}\,d\sigma=-\dfrac{1}{k+1}\int_{\brm_R^M}\Delta^M u_{k+1,R}d\sigma
			 \\&=\,-\dfrac{1}{k+1}\int_{\brm_R^M}\Delta^M\bu_{k+1,R}^{\omega}\,d\sigma= -\dfrac{1}{k+1}\,\bu_{k+1,R}^{\omega}{'}(R)\Vol\left(\srm_R^M\right)\\ &=\dfrac{\CA_k\left(\brm_R^\omega\right)}{\Vol\left(\srm_R^\omega\right)}\,\Vol\left(\srm_R^M\right)= \CA_k\left(\brm_R^\omega\right)
		\end{split}
	\end{equation}
	
	\noindent so we have inequalities \eqref{cor:5.2.2}. Now, we use equation (\ref{gim2}) to obtain \eqref{cor:5.2.1}.
	
\end{proof}

%%%%%%%%%%%%%%%%%%%%%%%%%
%%%%%%%%%%%%%%%%%%%%%%%%%%%%%%%%%%

\end{document}